\documentclass[12pt,bezier]{article}

\usepackage{amssymb,amsmath,amsthm,latexsym,tikz,url,subfig}
\newtheorem{theorem}{Theorem}[section]

\newtheorem{definition}[theorem]{Definition}
\newtheorem{conjecture}[theorem]{Conjecture}
\newtheorem{lemma}[theorem]{Lemma}
\newtheorem{remark}[theorem]{Remark}

\newcommand{\la}{\lambda}
\newcommand{\M}{{\cal M}}
\newcommand{\sw}{{\sf sw}}
\newcommand{\z}{{\bf z}}
\newcommand{\y}{{\bf y}}
\newcommand{\x}{{\bf x}}

\newcommand{\vertex}{\node[vertex]}
\tikzstyle{vertex}=[circle, draw, inner sep=0pt, minimum size=3pt]

\usepackage{color}
\newcommand{\w}{\color{blue}}

\begin{document}

\title{Regular Graphs with Minimum Spectral Gap}
\author{ M. Abdi$^{\,\rm a,b}$ \quad  E. Ghorbani$^{\,\rm a,b}$ \quad  W. Imrich$^{\,\rm c}$
\\[.3cm]
{\sl\normalsize $^{\rm a}$Department of Mathematics, K. N. Toosi University of Technology,}\\
{\sl\normalsize P. O. Box 16765-3381, Tehran, Iran}\\
{\sl\normalsize $^{\rm b}$School of Mathematics, Institute for Research in Fundamental Sciences (IPM),}\\
{\sl\normalsize P. O. Box 19395-5746, Tehran, Iran }\\
{\sl\normalsize $^{\rm c}$Montanuniversit\"at Leoben, Leoben, Austria}}

\maketitle
\footnotetext{{\em E-mail Addresses}: {\tt m.abdi@email.kntu.ac.ir} (M. Abdi), {\tt e\_ghorbani@ipm.ir} (E. Ghorbani), {\tt wilfried.imrich@unileoben.ac.at} (W. Imrich)}

\begin{abstract}
Aldous and Fill conjectured that the maximum relaxation time for the random walk on a connected regular graph with $n$ vertices is $(1+o(1)) \frac{3n^2}{2\pi^2}$.
 This conjecture can be rephrased in terms of the spectral gap as follows:
  the spectral gap (algebraic connectivity) of a connected  $k$-regular graph on $n$ vertices is at least
 $(1+o(1))\frac{2k\pi^2}{3n^2}$, and the bound is attained for at least one value of $k$.
 Based upon previous work of Brand, Guiduli, and Imrich, we prove this conjecture for cubic graphs.
   We also investigate  the structure of quartic (i.e.~4-regular) graphs with the minimum spectral gap among all connected quartic graphs.
   We show that they must have a path-like structure built from specific blocks.
   \vspace{4mm}

\noindent {\bf AMS Mathematics Subject Classification\,(2010):}   05C50, 60G50
\end{abstract}

\section{Introduction}

All graphs we consider are  simple, that is undirected graphs without loops or multiple edges.
The difference between the two largest eigenvalues of the adjacency matrix of a graph $G$ is called the {\em spectral gap} of $G$.
If $G$ is a regular graph, then its spectral gap is equal to the second smallest eigenvalue of its Laplacian matrix  and
known as {\em algebraic connectivity}.

 In $1976$,  Bussemaker, \v Cobelji\'c, Cvetkovi\'c, and Seidel (\cite{Bussemaker},  see also \cite{Bussemaker2}), by means of a computer search, found  all non-isomorphic connected cubic graphs with $n \leq 14$ vertices.  They observed that when the algebraic connectivity is small the graph is long. Indeed, as the algebraic connectivity decreases, both connectivity and girth decrease and diameter increases.  Based on these results, L. Babai (see \cite{Guiduli}) made a conjecture that described the structure of the connected cubic graph with minimum algebraic connectivity.
Guiduli \cite{Guiduli} (see also \cite{GuiduliThesis}) proved that the cubic graph with minimum algebraic connectivity must look like a path, built from specific blocks.
The result of Guiduli was improved as follows confirming the Babai's conjecture.

\begin{theorem}[Brand, Guiduli, and Imrich \cite{Imrich}]\label{thm:cubicBGI}
 Among all connected cubic graphs on $n$ vertices,  $n \geq 10$, the graph $G_n$ (given in Figure~\ref{fig:Gn}) is the unique graph with minimum algebraic connectivity.
 \end{theorem}
  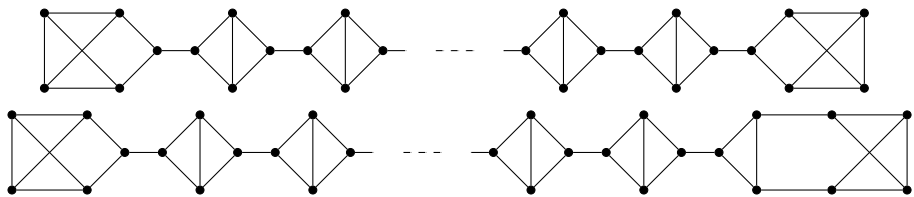
\begin{figure}[h!]
 \centering
\begin{tikzpicture}
	\vertex[fill] (1) at (0,-.5) [] {};
	\vertex[fill] (2) at (0,.5) [] {};
	\vertex[fill] (3) at (1,-.5) [] {};
	\vertex[fill] (4) at (1,.5) [] {};
    \vertex[fill] (5) at (1.5,0) [] {};
    \vertex[fill] (6) at (2,0) [] {};
    \vertex[fill] (7) at (2.5,.5) [] {};
    \vertex[fill] (8) at (2.5,-.5) [] {};
    \vertex[fill] (9) at (3,0) [] {};
    \vertex[fill] (10) at (3.5,0) [] {};
    \vertex[fill] (11) at (4,.5) [] {};
    \vertex[fill] (12) at (4,-.5) [] {};
    \vertex[fill] (13) at (4.5,0) [] {};
     \vertex[fill] (22) at (6.4,0) [] {};
    \vertex[fill] (23) at (6.9,.5) [] {};
    \vertex[fill] (24) at (6.9,-.5) [] {};
    \vertex[fill] (25) at (7.4,0) [] {};
    \vertex[fill] (26) at (7.9,0) [] {};
    \vertex[fill] (27) at (8.4,.5) [] {};
    \vertex[fill] (28) at (8.4,-.5) [] {};
    \vertex[fill] (29) at (8.9,0) [] {};
    \vertex[fill] (34) at (10.9,-.5) [] {};
	\vertex[fill] (33) at (10.9,.5) [] {};
	\vertex[fill] (32) at (9.9,-.5) [] {};
	\vertex[fill] (31) at (9.9,.5) [] {};
    \vertex[fill] (30) at (9.4,0) [] {};
    \tikzstyle{vertex}=[circle, draw, inner sep=0pt, minimum size=0pt]
    \vertex[fill] (14) at (4.8,0) [] {};
    \vertex[fill] (15) at (5.2,0) [] {};
    \vertex[fill] (16) at (5.3,0) [] {};
    \vertex[fill] (17) at (5.4,0) [] {};
    \vertex[fill] (18) at (5.5,0) [] {};
    \vertex[fill] (19) at (5.6,0) [] {};
    \vertex[fill] (20) at (5.7,0) [] {};
    \vertex[fill] (21) at (6.1,0) [] {};
	\path
		(1) edge (2)
		(1) edge (3)
	    (1) edge (4)
		(2) edge (3)
        (2) edge (4)
		(3) edge (5)
	    (4) edge (5)
	    (5) edge (6)
	    (6) edge (7)
	    (6) edge (8)
	    (7) edge (8)
	    (7) edge (9)
	    (8) edge (9)
	    (9) edge (10)
	    (10) edge (11)
	    (10) edge (12)
	    (11) edge (12)
	    (11) edge (13)
	    (12) edge (13)
	    (13) edge (14)
	   (15) edge (16)
	   (17) edge (18)
	   (19) edge (20)
	    (21) edge (22)
	    (22) edge (23)
	    (22) edge (24)
	    (23) edge (24)
	    (23) edge (25)
	    (24) edge (25)
	    (25) edge (26)
	    (26) edge (27)
	    (26) edge (28)
	    (27) edge (28)
	    (27) edge (29)
	    (28) edge (29)
	     (29) edge (30)
	      (30) edge (31)
	       (30) edge (32)
	        (31) edge (33)
	        (31) edge (34)
	       (32) edge (33)
	        (32) edge (34)
	         (33) edge (34);
\end{tikzpicture}
\vspace{.2cm} \\
\begin{tikzpicture}
	\vertex[fill] (1) at (0,-.5) [] {};
	\vertex[fill] (2) at (0,.5) [] {};
	\vertex[fill] (3) at (1,-.5) [] {};
	\vertex[fill] (4) at (1,.5) [] {};
    \vertex[fill] (5) at (1.5,0) [] {};
    \vertex[fill] (6) at (2,0) [] {};
    \vertex[fill] (7) at (2.5,.5) [] {};
    \vertex[fill] (8) at (2.5,-.5) [] {};
    \vertex[fill] (9) at (3,0) [] {};
    \vertex[fill] (10) at (3.5,0) [] {};
    \vertex[fill] (11) at (4,.5) [] {};
    \vertex[fill] (12) at (4,-.5) [] {};
    \vertex[fill] (13) at (4.5,0) [] {};
     \vertex[fill] (22) at (6.4,0) [] {};
    \vertex[fill] (23) at (6.9,.5) [] {};
    \vertex[fill] (24) at (6.9,-.5) [] {};
    \vertex[fill] (25) at (7.4,0) [] {};
    \vertex[fill] (26) at (7.9,0) [] {};
    \vertex[fill] (27) at (8.4,.5) [] {};
    \vertex[fill] (28) at (8.4,-.5) [] {};
    \vertex[fill] (29) at (8.9,0) [] {};
    \vertex[fill] (34) at (10.9,-.5) [] {};
	\vertex[fill] (33) at (10.9,.5) [] {};
	\vertex[fill] (32) at (9.9,-.5) [] {};
	\vertex[fill] (31) at (9.9,.5) [] {};
    \vertex[fill] (30) at (9.4,0) [] {};
   \vertex[fill] (35) at (11.9,-.5) [] {};
	\vertex[fill] (36) at (11.9,.5) [] {};
    \tikzstyle{vertex}=[circle, draw, inner sep=0pt, minimum size=0pt]
    \vertex[fill] (14) at (4.8,0) [] {};
    \vertex[fill] (15) at (5.2,0) [] {};
    \vertex[fill] (16) at (5.3,0) [] {};
    \vertex[fill] (17) at (5.4,0) [] {};
    \vertex[fill] (18) at (5.5,0) [] {};
    \vertex[fill] (19) at (5.6,0) [] {};
    \vertex[fill] (20) at (5.7,0) [] {};
    \vertex[fill] (21) at (6.1,0) [] {};
	\path
		(1) edge (2)
		(1) edge (3)
	    (1) edge (4)
		(2) edge (3)
        (2) edge (4)
		(3) edge (5)
	    (4) edge (5)
	    (5) edge (6)
	    (6) edge (7)
	    (6) edge (8)
	    (7) edge (8)
	    (7) edge (9)
	    (8) edge (9)
	    (9) edge (10)
	    (10) edge (11)
	    (10) edge (12)
	    (11) edge (12)
	    (11) edge (13)
	    (12) edge (13)
	    (13) edge (14)
	   (15) edge (16)
	   (17) edge (18)
	   (19) edge (20)
	    (21) edge (22)
	    (22) edge (23)
	    (22) edge (24)
	    (23) edge (24)
	    (23) edge (25)
	    (24) edge (25)
	    (25) edge (26)
	    (26) edge (27)
	    (26) edge (28)
	    (27) edge (28)
	    (27) edge (29)
	    (28) edge (29)
	     (29) edge (30)
	      (30) edge (31)
	       (30) edge (32)
	        (31) edge (32)
	        (31) edge (33)
	       (32) edge (34)
	        (33) edge (35)
	         (34) edge (36)
	        (33) edge (36)
	         (34) edge (35)
	         (35) edge (36);
\end{tikzpicture}
\caption{The cubic graph $G_n$,  $n \geq 10$, with minimum spectral gap on $n\equiv2\pmod4$ and $n\equiv0\pmod4$ vertices, respectively}
\label{fig:Gn}
\end{figure}

The {\em relaxation time} of the random walk on a graph $G$  is defined by $\tau=1/(1-\eta_2)$, where $\eta_2$ is the second largest eigenvalue of
the {\em transition matrix} of $G$, that is the matrix $D^{-1} A$ in which $D$ and $A$ are the diagonal matrix of vertex degrees and the adjacency matrix of $G$, respectively.
A central problem in the study of random walks is to determine the {\em mixing time}, a measure of how fast the
random walk converges to the stationary distribution. As seen throughout the literature \cite{aldous2002reversible,chung}, the  relaxation time is the primary term controlling mixing time. Therefore, relaxation time is directly associated with the rate of convergence  of the random walk.

Our main motivation in this work is the following conjecture on the maximum relaxation time of the random walk in regular graphs.

\begin{conjecture}[Aldous and Fill {\cite[p.~217]{aldous2002reversible}}]\rm
Over all connected regular graphs on $n$ vertices, $\max \tau =(1+o(1)) \frac{3n^2}{2\pi^2}$.
\end{conjecture}

In terms of the eigenvalues of the {\em normalized Laplacian matrix}, that is the matrix $I-D^{-1/2}AD^{-1/2}$, the Aldous--Fill conjecture says that the minimum second smallest eigenvalue of the normalized Laplacian matrices of all connected regular graphs on $n$ vertices is
 $(1+o(1))\frac{2\pi^2}{3n^2}$. This can be rephrased  in terms of the spectral gap as follows, giving another equivalent statement of the Aldous--Fill conjecture.

\begin{conjecture}\rm  The spectral gap (algebraic connectivity) of a connected  $k$-regular graph on $n$ vertices is at least
 $(1+o(1))\frac{2k\pi^2}{3n^2}$, and the bound is attained at least for one value of $k$.
\end{conjecture}

It is worth mentioning that in \cite{actt}, it is proved that the maximum relaxation time for the random walk on a  connected graph on $n$ vertices is $(1 +o(1))\frac{n^3}{54}$ settling another conjecture by Aldous and Fill (\cite[p.~216]{aldous2002reversible}).

    In \cite{Imrich}, it is mentioned without proof that the algebraic connectivity of the graphs $G_n$ (of Theorem~\ref{thm:cubicBGI}) is $(1+o(1))\frac{2\pi^2}{n^2}$, where its proof is postponed to another paper which has not appeared.
    We prove this equality,   thus, showing that  the minimum spectral gap  of connected cubic graphs on $n$ vertices is $(1+o(1))\frac{2\pi^2}{n^2}$, which implies the Aldous--Fill  conjecture for $k=3$.
     As the next case of the Aldous--Fill  conjecture and as a continuation of Babai's conjecture, we investigate the connected quartic, i.e.~$4$-regular, graphs with minimum spectral gap.
      We show that similar to the cubic case, these graphs must have a path-like structure with specified blocks (see Theorem~\ref{thm:quartic} below).
       Finally, we put forward a conjecture
       about the unique structure of the connected quartic graph  of any order with minimum spectral gap.

 \section{Minimum spectral gap of cubic graphs}
In this section, we prove that  the minimum spectral gap  of connected cubic graphs on $n$ vertices is $(1+o(1))\frac{2\pi^2}{n^2}$.

  Let $G$ be a graph on $n$ vertices and $L(G)=D-A$ be its Laplacian matrix. For any $\x\in \mathbb{R}^n$, the value
 $\frac{\x^\top L(G)\x}{\x^\top\x}$ is called a {\em Rayleigh quotient}.
 We denote the second smallest eigenvalue of $L(G)$ known as {\em the algebraic connectivity} of $G$  by $\mu(G)$. It is well known that
 \begin{equation}\label{eq}
 \mu(G)=\min_{\x\ne\bf0,\,\x\perp\bf1}\frac{\x^\top L(G)\x}{\x^\top\x},
 \end{equation}
 where $\bf1$ is the all-$1$ vector. An eigenvector corresponding to $\mu(G)$ is known as a {\em Fiedler vector} of $G$.
In passing we note that if $\x=(x_1,\ldots,x_n)^\top$, then
 $$\x^\top L(G)\x=\sum_{ij\in E(G)}(x_i-x_j)^2,$$
 where $E(G)$ is the edge set of $G$.

Considering the graphs $G_n$ of Theorem~\ref{thm:cubicBGI}, we  let $\Pi=\{C_1, C_2, \ldots , C_k\}$ (numbered consecutively from left to right) be a partition of the vertex set $V(G_n)$ such that each cell $C_i$ has size 1 or 2,  consisting of the vertices drawn vertically above each other as depicted in Figure~\ref{fig:Gn}.  
 We note in passing that partition $\Pi$ is a so-called `equitable partition' of $G_n$.

  \begin{lemma}[\cite{Imrich}] \label{lem:sign}
 Let $\x$ be a Fiedler vector of $G_n$.
 \begin{itemize}
   \item[\rm(i)] Then the components of $\x$ on each cell $C_i$ of the partition $\Pi$ are equal.
   \item[\rm(ii)]  Let $x_1, \ldots, x_k$ be the values of $\x$ on the cells of $\Pi$. Then  the $x_i$ form a strictly monotone sequence changing sign once.
 \end{itemize}
  \end{lemma}

Recall that a {\em block} of a graph is a maximal connected subgraph
with no cut vertex|a subgraph with as many edges as possible and no cut vertex. So
a block is either $K_2$ (a trivial block) or is a graph which contains a cycle. If a graph $G$ has no cut vertex, then $G$ itself is also called a block.
The blocks of a connected graph fit together in a tree-like
structure, called the {\em block tree} of $G$. The block tree of the graphs $G_n$ are paths which justifies the description `path-like structure.'

We now present the main result of this section.

\begin{theorem}\label{thm:cubic}
The  minimum algebraic connectivity of  cubic graphs on $n$ vertices is $(1+o(1))\frac{2\pi^2}{n^2}$.
 \end{theorem}
\begin{proof}{In view of Theorem~\ref{thm:cubicBGI}, it suffices to show that $\mu(G_n)=(1+o(1))\frac{2\pi^2}{n^2}$.
To prove this, we consider two cases based on the value of $n$ mod $4$. 

\noindent{\bf Case 1.} $n\equiv2\pmod4$

In this case $G_n$ is the upper graph of Figure~\ref{fig:Gn}.
 Let $m+2$ be the number of non-trivial blocks   of $G_n$. So we have  $n=4m+10$.

 We first prove that $(1+o(1))\frac{2\pi^2}{n^2}$ is an upper bound for $\mu(G_n)$.

  We define  the vector $\x=(x_1,\ldots,x_{2m})^\top$ with
 $$x_i=\cos\left(\frac{(2i-1)\pi}{4m}\right),~~i=1,\ldots,2m.$$
 Note that  $\x$ is  skew symmetric vector,  i.e. $x_{2m-i+1}=-x_i$, for $i=1,\ldots,m$, and so
 $\x\perp\bf1$. We extend $\x$ to define the vector $\x'$ on $G_n$ as shown in Figure~\ref{fig:cubic1}.

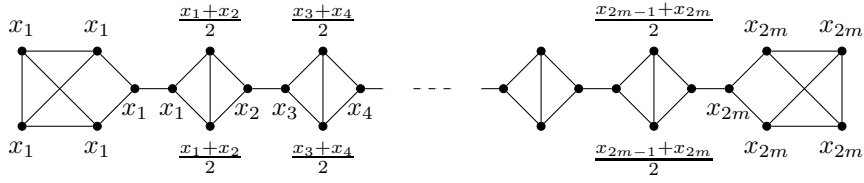
\begin{figure}[h!]
\centering
\begin{tikzpicture}
	\vertex[fill] (1) at (0,-.5) [label=below:\footnotesize{$x_1$}]{};
	\vertex[fill] (2) at (0,.5) [label=above:\footnotesize{$x_1$}] {};
	\vertex[fill] (3) at (1,-.5) [label=below:\footnotesize{$x_1$}] {};
	\vertex[fill] (4) at (1,.5) [label=above:\footnotesize{$x_1$}] {};
    \vertex[fill] (5) at (1.5,0) [label=below:\footnotesize{$x_1$}] {};
    \vertex[fill] (6) at (2,0) [label=below:\footnotesize{$x_1$}] {};
    \vertex[fill] (7) at (2.5,.5) [label=above:\footnotesize{$\frac{x_1+x_2}{2}$}] {};
    \vertex[fill] (8) at (2.5,-.5) [label=below:\footnotesize{$\frac{x_1+x_2}{2}$}] {};
    \vertex[fill] (9) at (3,0) [label=below:\footnotesize{$x_2$}] {};
    \vertex[fill] (10) at (3.5,0) [label=below:\footnotesize{$x_3$}] {};
    \vertex[fill] (11) at (4,.5) [label=above:\footnotesize{$\frac{x_3+x_4}{2}$}] {};
    \vertex[fill] (12) at (4,-.5) [label=below:\footnotesize{$\frac{x_3+x_4}{2}$}] {};
    \vertex[fill] (13) at (4.5,0) [label=below:\footnotesize{$x_4$}] {};
     \vertex[fill] (22) at (6.4,0) [] {};
    \vertex[fill] (23) at (6.9,.5) [] {};
    \vertex[fill] (24) at (6.9,-.5) [] {};
    \vertex[fill] (25) at (7.4,0) [] {};
    \vertex[fill] (26) at (7.9,0) [] {};
    \vertex[fill] (27) at (8.4,.5) [label=above:\footnotesize{$\frac{x_{2m-1}+x_{2m}}{2}$}] {};
    \vertex[fill] (28) at (8.4,-.5) [label=below:\footnotesize{$\frac{x_{2m-1}+x_{2m}}{2}$}] {};
    \vertex[fill] (29) at (8.9,0) [] {};
    \vertex[fill] (34) at (10.9,-.5) [label=below:\footnotesize{$x_{2m}$}] {};
	\vertex[fill] (33) at (10.9,.5) [label=above:\footnotesize{$x_{2m}$}] {};
	\vertex[fill] (32) at (9.9,-.5) [label=below:\footnotesize{$x_{2m}$}] {};
	\vertex[fill] (31) at (9.9,.5) [label=above:\footnotesize{$x_{2m}$}] {};
    \vertex[fill] (30) at (9.4,0) [label=below:\footnotesize{$x_{2m}$}] {};
    \tikzstyle{vertex}=[circle, draw, inner sep=0pt, minimum size=0pt]
    \vertex[fill] (14) at (4.8,0)[]{} ;
    \vertex[fill] (15) at (5.2,0)[]{} ;
    \vertex[fill] (16) at (5.3,0)[]{} ;
    \vertex[fill] (17) at (5.4,0)[]{} ;
    \vertex[fill] (18) at (5.5,0)[]{} ;
    \vertex[fill] (19) at (5.6,0)[]{} ;
    \vertex[fill] (20) at (5.7,0)[]{} ;
    \vertex[fill] (21) at (6.1,0)[]{} ;
	\path
		(1) edge (2)
		(1) edge (3)
	    (1) edge (4)
		(2) edge (3)
        (2) edge (4)
		(3) edge (5)
	    (4) edge (5)
	    (5) edge (6)
	    (6) edge (7)
	    (6) edge (8)
	    (7) edge (8)
	    (7) edge (9)
	    (8) edge (9)
	    (9) edge (10)
	    (10) edge (11)
	    (10) edge (12)
	    (11) edge (12)
	    (11) edge (13)
	    (12) edge (13)
	    (13) edge (14)
	   (15) edge (16)
	   (17) edge (18)
	   (19) edge (20)
	    (21) edge (22)
	    (22) edge (23)
	    (22) edge (24)
	    (23) edge (24)
	    (23) edge (25)
	    (24) edge (25)
	    (25) edge (26)
	    (26) edge (27)
	    (26) edge (28)
	    (27) edge (28)
	    (27) edge (29)
	    (28) edge (29)
	     (29) edge (30)
	      (30) edge (31)
	       (30) edge (32)
	        (31) edge (33)
	        (31) edge (34)
	       (32) edge (33)
	        (32) edge (34)
	         (33) edge (34);
\end{tikzpicture}
 \caption{The components of $\x'$ on $G_n$, $n\equiv2\pmod4$}\label{fig:cubic1}
\end{figure}
The vector $\x'$ (like $\x$) is a skew symmetric. It follows that $\x'\perp\bf1$. Therefore, by \eqref{eq} we have
\begin{align}
\mu(G_n)&\le\frac{\x'^\top L(G_n)\x'}{\x'^\top\x'}\nonumber\\
&\leq \frac{\sum_{i=1}^{2m-1}(x_i-x_{i+1})^2}{\sum_{i=1}^{2m}x_i^2+2\sum_{i=1}^{m}\frac{1}{4}(x_{2i-1}+x_{2i})^2+10x_1^2}\nonumber \\
&\leq \frac{4\sin^2(\frac{\pi}{4m})\sum_{i=1}^{2m-1}\sin^2(\frac{\pi  i}{2m})}{\sum_{i=1}^{2m}\cos^2(\frac{(2i-1)\pi}{4m})
+2\cos^2(\frac{\pi}{4m})\sum_{i=1}^m\cos^2(\frac{(2i-1)\pi}{2m})}\label{eq:cos-sin}\\
&=\frac{4m \sin^2(\frac{\pi}{4m})}{m+m\cos^2(\frac{\pi}{4m})}\label{eq:sin^2} \\
&=(1+o(1))\frac{2\pi^2}{n^2}.\nonumber
\end{align}
Note that \eqref{eq:cos-sin} is obtained using the identities $\cos\alpha-\cos\beta=-2 \sin \frac{\alpha+\beta}{2} \sin \frac{\alpha-\beta}{2}$ and $\cos\alpha+\cos\beta=2 \cos \frac{\alpha+\beta}{2} \cos \frac{\alpha-\beta}{2}$.  For \eqref{eq:sin^2} we use the identities
$$\sum_{i=1}^{2m-1}\sin^2\left(\frac{\pi i}{2m}\right)=\sum_{i=1}^{2m}\cos^2\left(\frac{(2i-1)\pi}{4m}\right)=m,~~~
\sum_{i=1}^m\cos^2\left(\frac{(2i-1)\pi}{2m}\right)=\frac m2$$
which  are  a consequence of the fact that
$\sin^2(\alpha)+\sin^2(\frac\pi2-\alpha)=\cos^2(\alpha)+\cos^2(\frac\pi2-\alpha)=1$.

We now prove that $(1+o(1))\frac{2\pi^2}{n^2}$ is a lower bound for $\mu(G_n)$.

Let $\y=(y_1, y_2, \ldots, y_n)^\top$ be a Fiedler vector of $G_n$.
 Let $B_1,\ldots,B_{m+2}$ be the non-trivial blocks of $G_n$, and
   $E_1$ be the set of edges of $B_1,\ldots,B_{m+2}$ and $E_2$ be the set of all bridges of $G_n$. Then we have
 \begin{align} \label{mu}
\mu(G_n) &=\frac{\y^\top L(G_n)\y}{\y^\top\y}\nonumber\\
&=\frac{\sum_{ij\in E(G_n)}(y_i-y_j)^2}{\sum_{i=1}^ny_i^2}\nonumber\\
&= \frac{\sum_{ij\in E_1}(y_i-y_j)^2+\sum_{ij\in E_2}(y_i-y_j)^2}{\sum_{i=1}^ny_i^2}.
 \end{align}
 The graph $G_n$ has $2m+2$ cut vertices. Consider the components of $\y$ on the cut vertices of $G_n$ together with the four components $y_1, y_3, y_{n-2},y_n$;
we define  $\z$ as the vector consisting of these $2m+6$ components, as depicted in Figure~\ref{fig:cubic2a}.

 \begin{figure}[h!]
 \centering
 \begin{tikzpicture}
	\vertex[fill] (1) at (0,-.5) [] {};
	\vertex[fill] (2) at (0,.5) [label=above:\footnotesize{$z_1$}] {};
	\vertex[fill] (3) at (1,-.5) [] {};
	\vertex[fill] (4) at (1,.5) [label=above:\footnotesize{$z_2$}] {};
    \vertex[fill] (5) at (1.5,0) [label=above:\footnotesize{$z_3$}] {};
    \vertex[fill] (6) at (2,0) [label=above:\footnotesize{$z_4$}] {};
    \vertex[fill] (7) at (2.5,.5) [] {};
    \vertex[fill] (8) at (2.5,-.5) [] {};
    \vertex[fill] (9) at (3,0) [label=above:\footnotesize{$z_5$}] {};
    \vertex[fill] (10) at (3.5,0) [label=above:\footnotesize{$z_6$}] {};
    \vertex[fill] (11) at (4,.5) [] {};
    \vertex[fill] (12) at (4,-.5) [] {};
    \vertex[fill] (13) at (4.5,0) [label=above:\footnotesize{$z_7$}] {};
     \vertex[fill] (22) at (6.4,0) [] {};
    \vertex[fill] (23) at (6.9,.5) [] {};
    \vertex[fill] (24) at (6.9,-.5) [] {};
    \vertex[fill] (25) at (7.4,0) [] {};
    \vertex[fill] (26) at (7.9,0) [] {};
    \vertex[fill] (27) at (8.4,.5) [] {};
    \vertex[fill] (28) at (8.4,-.5) [] {};
    \vertex[fill] (29) at (8.9,0) [] {};
    \vertex[fill] (34) at (10.9,-.5) [] {};
	\vertex[fill] (33) at (10.9,.5) [label=above:\footnotesize{$z_{2m+6}$}] {};
	\vertex[fill] (32) at (9.9,-.5) [] {};
	\vertex[fill] (31) at (9.9,.5) [label=above:\footnotesize{$z_{2m+5}$}] {};
    \vertex[fill] (30) at (9.4,0) [] {};
    \tikzstyle{vertex}=[circle, draw, inner sep=0pt, minimum size=0pt]
    \vertex[fill] (14) at (4.8,0) [] {};
    \vertex[fill] (15) at (5.2,0) [] {};
    \vertex[fill] (16) at (5.3,0) [] {};
    \vertex[fill] (17) at (5.4,0) [] {};
    \vertex[fill] (18) at (5.5,0) [] {};
    \vertex[fill] (19) at (5.6,0) [] {};
    \vertex[fill] (20) at (5.7,0) [] {};
    \vertex[fill] (21) at (6.1,0) [] {};
	\path
		(1) edge (2)
		(1) edge (3)
	    (1) edge (4)
		(2) edge (3)
        (2) edge (4)
		(3) edge (5)
	    (4) edge (5)
	    (5) edge (6)
	    (6) edge (7)
	    (6) edge (8)
	    (7) edge (8)
	    (7) edge (9)
	    (8) edge (9)
	    (9) edge (10)
	    (10) edge (11)
	    (10) edge (12)
	    (11) edge (12)
	    (11) edge (13)
	    (12) edge (13)
	    (13) edge (14)
	   (15) edge (16)
	   (17) edge (18)
	   (19) edge (20)
	    (21) edge (22)
	    (22) edge (23)
	    (22) edge (24)
	    (23) edge (24)
	    (23) edge (25)
	    (24) edge (25)
	    (25) edge (26)
	    (26) edge (27)
	    (26) edge (28)
	    (27) edge (28)
	    (27) edge (29)
	    (28) edge (29)
	     (29) edge (30)
	      (30) edge (31)
	       (30) edge (32)
	        (31) edge (33)
	        (31) edge (34)
	       (32) edge (33)
	        (32) edge (34)
	         (33) edge (34);
\end{tikzpicture}
\caption{The vector $\z$ defined on cut vertices and end blocks of $G_n$}\label{fig:cubic2a}
\end{figure}
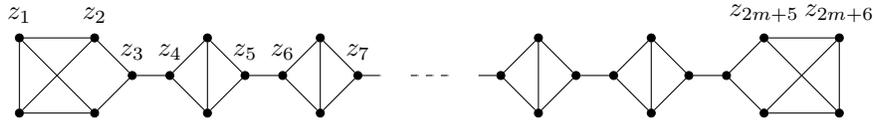

Note  that $\y$ is skew symmetric.
To verify this, observe that by the symmetry of $G_n$, $\y'=(y_n,y_{n-1},\ldots,y_1)$ is also an eigenvector for $\mu(G_n)$.
It follows that $\y-\y'$ itself is a skew symmetric eigenvector for $\mu(G_n)$ (note that from Lemma~\ref{lem:sign}, it is seen that $\y-\y'\ne\bf0$), so that we may replace $\y-\y'$ for $\y$. Now, from Lemma~\ref{lem:sign}, it follows that $\z=(z_1, z_2, \ldots, z_{2m+6})\ne\bf0$.
As $\y$ is skew symmetric, it follows that $\z$ is also skew symmetric and thus $\z\perp\bf1$.
 Let $B_k$ be one of the middle blocks of $G_n$, i.e. $2\le k\le m+1$.
 The components of $\y$ on the left vertex and the right vertex of $B_k$  are $z_{2k}$ and $z_{2k+1}$, respectively.
 Let $t$ be the component of $\y$ on the two middle vertices of $B_k$ (which are equal by Lemma~\ref{lem:sign}) as shown in Figure~\ref{fig:cubic2b}.
 \begin{figure}[h!]
 \centering
 \begin{tikzpicture}
    \vertex[fill] (6) at (2,0) [label=left:\footnotesize{$z_{2k}$}] {};
    \vertex[fill] (7) at (2.5,.5) [label=above:\footnotesize{$t$}] {};
    \vertex[fill] (8) at (2.5,-.5) [label=below:\footnotesize{$t$}] {};
    \vertex[fill] (9) at (3,0) [label=right:\footnotesize{$z_{2k+1}$}] {};
	\path
	    (6) edge (7)
	    (6) edge (8)
	    (7) edge (8)
	    (7) edge (9)
	    (8) edge (9);
\end{tikzpicture}
\caption{The components of $\y$ on a middle block $B_k$}\label{fig:cubic2b}
\end{figure}
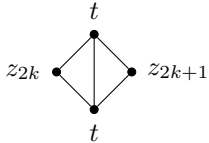

 Then
 $$\sum_{ij\in E(B_k)}(y_i-y_j)^2=2(z_{2k}-t)^2+2(t-z_{2k+1})^2.$$
 The right hand side, considered as a function of $t$,  is minimized at $t=\frac12(z_{2k}+z_{2k+1})$. This implies that
$$\sum_{ij\in E(B_k)}(y_i-y_j)^2\ge(z_{2k}-z_{2k+1})^2.$$
  It follows that
\begin{align*}
\sum_{ij\in E_1}(y_i-y_j)^2 &= \sum_{ij\in E(B_1)}(y_i-y_j)^2+\sum_{k=2}^{m+1}\sum_{ij\in E(B_k)}(y_i-y_j)^2 +\sum_{ij\in E(B_{m+2})}(y_i-y_j)^2 \\
&\geq 4(z_1-z_2)^2+2(z_2-z_3)^2+\sum_{k=2}^{m+1} (z_{2k}-z_{2k+1})^2\\&~~~ +2(z_{2m+4}-z_{2m+5})^2+4(z_{2m+5}-z_{2m+6})^2  \\
&\geq (z_1-z_2)^2+\sum_{k=1}^{m+2} (z_{2k}-z_{2k+1})^2+(z_{2m+5}-z_{2m+6})^2,
\end{align*}
which in turn implies that
\begin{equation}\label{u}
\sum_{ij\in E_1}(y_i-y_j)^2+\sum_{ij\in E_2}(y_i-y_j)^2\geq \sum_{r=1}^{2m+5}(z_r-z_{r+1})^2.
\end{equation}
We also have
\begin{equation} \label{d}
\sum_{i=1}^ny_i^2 \leq 2\sum_{i=1}^{2m+6}z_i^2,
\end{equation}
which holds because $y_1^2+y_2^2=2z_1^2, y_3^2+y_4^2=2z_2^2, y_5^2+y_7^2\leq 2z_3^2,  y_6^2+y_8^2\leq 2z_4^2, \ldots, y_{n-4}^2+y_{n-7}^2\leq 2z_{2m+4}^2, \ldots$ (cf. Lemma~\ref{lem:sign}).
Now, from \eqref{mu}, \eqref{u} and \eqref{d} we  infer that
\begin{equation} \label{eq:P_{2m+6}}
\mu(G_n)\geq \frac{\sum_{i=1}^{2m+5}(z_i-z_{i+1})^2}{2\sum_{i=1}^{2m+6}z_i^2}.
\end{equation}
Note that the right hand side of \eqref{eq:P_{2m+6}} is the Rayleigh quotient of $\z$ for the path $P_{2m+6}$.
Thus, by the fact that  $\mu(P_h)=2(1-\cos\frac{\pi}{h})$ (see \cite{fiedler1973algebraic}),
it follows that
\[\frac{\sum_{i=1}^{2m+5}(z_i-z_{i+1})^2}{\sum_{i=1}^{2m+6}z_i^2}\ge\mu(P_{2m+6})=(1+o(1))\frac{\pi^2}{4m^2}.\]
Therefore,
\[\mu(G_n)\geq (1+o(1))\frac{2\pi^2}{n^2}.\]

\noindent{\bf Case 2.} $n\equiv0\pmod4$

 In this case, $G_n$ is the bottom graph of Figure~\ref{fig:Gn}. We define the graph $H_{n+2}$ as shown in Figure~\ref{fig:cubic3}.
 The symmetries of $H_{n+2}$ are similar to  those of  the graph $G_{n-2}$. So the arguments of the previous case also work for $H_{n+2}$, in particular $H_{n+2}$ has a skew symmetric Fiedler vector. Therefore, we have $\mu(H_{n+2})= (1+o(1))\frac{2\pi^2}{(n+2)^2}$.
 Let  $\x=(x_1,  \ldots, x_n)^\top$ be the Fiedler vector of  $G_n$ with
 $\|\x\|=1$.
 We define the vector $\y$ of length $n+2$ by
$$y_i= \begin{cases}
x_i-\delta  & i=1, 2, 3, 4, \\
 x_5-\delta & i=5, 6,  \\
 x_{i-2}-\delta& i=7,\ldots,n+2,
\end{cases}$$
  where $\delta=\frac{2x_5}{n+2}$.  It is seen that $\y$ is orthogonal to $\bf1$.
  We  label the vertices of $H_{n+2}$ by the components of  $\y$   as shown in Figure~\ref{fig:cubic3}.
 We observe that
  $\sum_{ij\in E(G_n)} (x_i-x_j)^2=\sum_{ij\in E(H_{n+2})} (y_i-y_j)^2$.
 On the other hand,
 \begin{align*}
 \|\y\|^2&=\sum_{i=1}^{n+2} y_i^2 \\
 &=\sum_{i=1}^n (x_i-\delta)^2 +2(x_5-\delta)^2 \\
 &= \sum_{i=1}^n x_i^2-2\delta \sum_{i=1}^n x_i+n\delta^2+2(x_5-\delta)^2 \\
 &= 1+2x_5^2\left(1-\frac{2}{n+2}\right).
 \end{align*}
 So $\|{\bf{y}}\|>1$, which means that the Rayleigh quotient for $\y$ on $H_{n+2}$  is smaller than $\mu(G_n)$.
 It follows that  $(1+o(1))\frac{2\pi^2}{(n+2)^2}=\mu(H_{n+2})\leq \mu(G_n)$.
  By a similar argument,  we see that $\mu(G_n)\leq \mu(G_{n-2})=(1+o(1))\frac{2\pi^2}{(n-2)^2}$. Therefore,   $\mu(G_n)= (1+o(1))\frac{2\pi^2}{n^2}$.
  }\end{proof}
   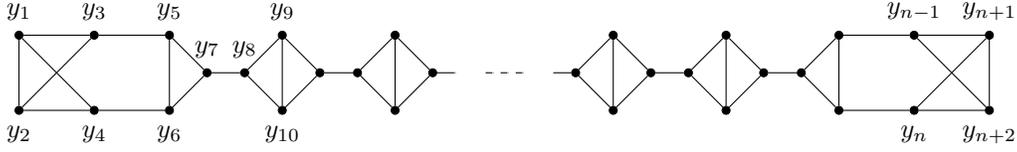
\begin{figure}
\centering
\begin{tikzpicture}
	\vertex[fill] (2) at (-1,-.5) [label=below:\footnotesize{$y_2$}] {};
	\vertex[fill] (1) at (-1,.5) [label=above:\footnotesize{$y_1$}] {};
	\vertex[fill] (4) at (0,-.5) [label=below:\footnotesize{$y_4$}] {};
	\vertex[fill] (3) at (0,.5) [label=above:\footnotesize{$y_3$}] {};
    \vertex[fill] (0) at (1,-.5) [label=below:\footnotesize{$y_6$}] {};
	\vertex[fill] (00) at (1,.5) [label=above:\footnotesize{$y_5$}] {};
    \vertex[fill] (5) at (1.5,0) [label=above:\footnotesize{$y_7$}] {};
    \vertex[fill] (6) at (2,0) [label=above:\footnotesize{$y_8$}] {};
    \vertex[fill] (7) at (2.5,.5) [label=above:\footnotesize{$y_9$}] {};
    \vertex[fill] (8) at (2.5,-.5) [label=below:\footnotesize{$y_{10}$}] {};
    \vertex[fill] (9) at (3,0) [] {};
    \vertex[fill] (10) at (3.5,0) [] {};
    \vertex[fill] (11) at (4,.5) [] {};
    \vertex[fill] (12) at (4,-.5) [] {};
    \vertex[fill] (13) at (4.5,0) [] {};
     \vertex[fill] (22) at (6.4,0) [] {};
    \vertex[fill] (23) at (6.9,.5) [] {};
    \vertex[fill] (24) at (6.9,-.5) [] {};
    \vertex[fill] (25) at (7.4,0) [] {};
    \vertex[fill] (26) at (7.9,0) [] {};
    \vertex[fill] (27) at (8.4,.5) [] {};
    \vertex[fill] (28) at (8.4,-.5) [] {};
    \vertex[fill] (29) at (8.9,0) [] {};
    \vertex[fill] (34) at (10.9,-.5) [label=below:\footnotesize{$y_n$}] {};
	\vertex[fill] (33) at (10.9,.5) [label=above:\footnotesize{$y_{n-1}$}] {};
	\vertex[fill] (32) at (9.9,-.5) [] {};
	\vertex[fill] (31) at (9.9,.5) [] {};
    \vertex[fill] (30) at (9.4,0) [] {};
   \vertex[fill] (35) at (11.9,-.5) [label=below:\footnotesize{$y_{n+2}$}] {};
	\vertex[fill] (36) at (11.9,.5) [label=above:\footnotesize{$y_{n+1}$}] {};
    \tikzstyle{vertex}=[circle, draw, inner sep=0pt, minimum size=0pt]
    \vertex[fill] (14) at (4.8,0) [] {};
    \vertex[fill] (15) at (5.2,0) [] {};
    \vertex[fill] (16) at (5.3,0) [] {};
    \vertex[fill] (17) at (5.4,0) [] {};
    \vertex[fill] (18) at (5.5,0) [] {};
    \vertex[fill] (19) at (5.6,0) [] {};
    \vertex[fill] (20) at (5.7,0) [] {};
    \vertex[fill] (21) at (6.1,0) [] {};
	\path
		(1) edge (2)
		(1) edge (3)
	    (1) edge (4)
	     (2) edge (3)
	    (2) edge (4)
		(3) edge (00)
        (4) edge (0)
        (0) edge (00)
		(0) edge (5)
	    (00) edge (5)
	    (5) edge (6)
	    (6) edge (7)
	    (6) edge (8)
	    (7) edge (8)
	    (7) edge (9)
	    (8) edge (9)
	    (9) edge (10)
	    (10) edge (11)
	    (10) edge (12)
	    (11) edge (12)
	    (11) edge (13)
	    (12) edge (13)
	    (13) edge (14)
	   (15) edge (16)
	   (17) edge (18)
	   (19) edge (20)
	    (21) edge (22)
	    (22) edge (23)
	    (22) edge (24)
	    (23) edge (24)
	    (23) edge (25)
	    (24) edge (25)
	    (25) edge (26)
	    (26) edge (27)
	    (26) edge (28)
	    (27) edge (28)
	    (27) edge (29)
	    (28) edge (29)
	     (29) edge (30)
	      (30) edge (31)
	       (30) edge (32)
	        (31) edge (32)
	        (31) edge (33)
	       (32) edge (34)
	        (33) edge (35)
	         (34) edge (36)
	        (33) edge (36)
	         (34) edge (35)
	         (35) edge (36);
\end{tikzpicture}
\caption{The graph $H_{n+2}$ and the components of $\y$}\label{fig:cubic3}
\end{figure}

\section{Structure of quartic graphs with minimum spectral gap}

Motivated by the Aldous--Fill Conjecture and also as an analogue  to Babai's conjecture,
 we consider the problem of determining the structure of connected quartic graphs with minimum spectral gap.
We prove that such graphs  have a path-like structure (see Figure~\ref{fig:path-like}) and specify their blocks. Finally, we pose a conjecture which precisely describes the connected quartic graphs with minimum spectral gap.

We remark that in a quartic graph, any cut vertex belongs to exactly two blocks and further has degree $2$ in each of them.
 Therefore, in the quartic graphs having a path-like structure, the middle and end blocks have exactly two and one vertices of degree $2$, respectively.

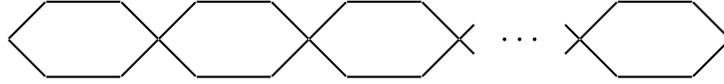
\begin{figure}[h!]
	\centering
	\begin{tikzpicture}
	\tikzstyle{vertex}=[draw, inner sep=0pt, minimum size=0pt]
	\vertex[fill] (1) at (0,0) [label=left:\tiny{}] {};
	\vertex[fill] (2) at (.5,.5) [] {};
	\vertex[fill] (3) at (.5,-.5) [] {};
	\vertex[fill] (4) at (1.5,.5) [] {};
	\vertex[fill] (5) at (1.5,-.5) [] {};
	\vertex[fill] (6) at (2,0) [] {};
	\vertex[fill] (7) at (2.5,.5) [] {};
	\vertex[fill] (8) at (2.5,-.5) [] {};
	\vertex[fill] (9) at (3.5,.5) [] {};
	\vertex[fill] (10) at (3.5,-.5) [] {};
	\vertex[fill] (11) at (4,0) [] {};
	\vertex[fill] (12) at (4.5,.5) [] {};
	\vertex[fill] (13) at (4.5,-.5) [] {};
	\vertex[fill] (14) at (5.5,.5) [] {};
	\vertex[fill] (15) at (5.5,-.5) [] {};
	\vertex[fill] (16) at (6,0) [] {};
	\vertex[fill] (20) at (7.6,0) [] {};
	\vertex[fill] (21) at (8.1,.5) [] {};
	\vertex[fill] (22) at (8.1,-.5) [] {};
	\vertex[fill] (23) at (9.1,.5) [] {};
	\vertex[fill] (24) at (9.1,-.5) [] {};
	\vertex[fill] (25) at (9.6,0) [] {};
	\tikzstyle{vertex}=[circle, draw, inner sep=0pt, minimum size=1pt]
	\vertex[fill] (17) at (6.6,0) [] {};
	\vertex[fill] (18) at (6.8,0) [] {};
	\vertex[fill] (19) at (7,0) [] {};
	\tikzstyle{vertex}=[circle, draw, inner sep=0pt, minimum size=0pt]
	\vertex (s) at (6.2,.2) [label=right:$$] {};
	\vertex (ss) at (6.2,-.2) [label=right:$$] {};
	\vertex (sss) at (7.4,.2) [label=right:$$] {};
	\vertex (ssss) at (7.4,-.2) [label=right:$$] {};
	\path[draw,thick]
	(1) edge (2)
	(1) edge (3)
	(3) edge (5)
	(2) edge (4)
	(4) edge (6)
	(5) edge (6)
	(6) edge (7)
	(6) edge (8)
	(7) edge (9)
	(8) edge (10)
	(10) edge (11)
	(9) edge (11)
	(11) edge (12)
	(11) edge (13)
	(12) edge (14)
	(13) edge (15)
	(14) edge (16)
	(15) edge (16)
	(16) edge (ss)
	(16) edge (s)
	(sss) edge (20)
	(ssss) edge (20)
	(20) edge (21)
	(20) edge (22)
	(21) edge (23)
	(22) edge (24)
	(24) edge (25)
	(23) edge (25);
	\end{tikzpicture}
	\caption{The path-like structure}\label{fig:path-like}
	\end{figure}

One of our goals in this section is to specify the structure of the blocks of a quartic graph with minimum spectral gap.
As we shall prove, the blocks of such graphs are
of two types: `short' and `long'.
By {\em short} blocks we mean those given in Figure~\ref{fig:short}.
  \begin{figure}[h!]
\captionsetup[subfigure]{labelformat=empty}
\centering
\subfloat[$M$]{\begin{tikzpicture}[scale=.9]
       \vertex[fill] (r) at (1,0) [] {};
 		\vertex[fill] (r1) at (1.5,.5) [] {};
 		\vertex[fill] (r2) at (1.5,-.5)[] {};
 		\vertex[fill] (r3) at (2,0) [] {};
 		\vertex[fill] (r4) at (2.5,.5) [] {};
 		\vertex[fill] (r5) at (2.5,-.5) [] {};
 	\vertex[fill] (r6) at (3,0) [] {};
 		\path
 	   (r) edge (r1)
 		(r) edge (r2)
 		(r1) edge (r2)
 		(r1) edge (r3)
 		(r1) edge (r4)
 		(r2) edge (r3)
 		(r2) edge (r5)
 		(r3) edge (r4)
 		(r3) edge (r5)
 		(r5) edge (r4)
 	     (r6) edge (r5)
 		(r6) edge (r4) 	;
 		\end{tikzpicture}}
 \qquad
\subfloat[$M_1$]{\begin{tikzpicture}[scale=.9]
   \vertex[fill] (r) at (0,0) [] {};
	\vertex[fill] (r1) at (.5,.5) [] {};
	\vertex[fill] (r2) at (.5,-.5) [] {};
	\vertex[fill] (r3) at (1.5,.5) [] {};
	\vertex[fill] (r4) at (1.5,-.5) [] {};
\vertex[fill] (r5) at (2,0) [] {};
	\path
   (r) edge (r1)
    (r) edge (r2)
	(r1) edge (r2)
	(r1) edge (r3)
	(r1) edge (r4)
	(r2) edge (r3)
	(r2) edge (r4)
	(r3) edge (r4)
    (r3) edge (r5)
	(r4) edge (r5);
	\end{tikzpicture}}
\qquad
\subfloat[$M_2$]{\begin{tikzpicture}[scale=.9]
\vertex[fill] (r) at (0,0) [] {};
	\vertex[fill] (r1) at (.5,.5) [] {};
	\vertex[fill] (r2) at (.5,-.5) [] {};
	\vertex[fill] (r3) at (1.5,.5) [] {};
	\vertex[fill] (r4) at (1.5,-.5) [] {};
	\vertex[fill] (r5) at (2.5,.5) [] {};
	\vertex[fill] (r6) at (2.5,-.5) [] {};
\vertex[fill] (r7) at (3,0) [] {};
	\path
    ( r) edge (r1)
	(r) edge (r2)
	(r5) edge (r6)
	(r1) edge (r2)
	(r1) edge (r3)
	(r1) edge (r4)
	(r2) edge (r3)
	(r2) edge (r4)
	(r3) edge (r6)
	(r4) edge (r5)
	(r5) edge (r3)
	(r4) edge (r6)
    (r5) edge (r7)
	(r7) edge (r6);
	\end{tikzpicture}}
\qquad
\subfloat[$M_3$]{\begin{tikzpicture}[scale=.9]
    \vertex[fill] (r) at (1,0) [] {};
	\vertex[fill] (r1) at (1.5,.5) [] {};
	\vertex[fill] (r2) at (1.5,-.5) [] {};
	\vertex[fill] (r3) at (2,0) [] {};
	\vertex[fill] (r4) at (2.5,.5) [] {};
	\vertex[fill] (r5) at (2.5,-.5) [] {};
	\vertex[fill] (r6) at (3.5,.5) [] {};
	\vertex[fill] (r7) at (3.5,-.5) [] {};
   \vertex[fill] (r8) at (4,0) [] {};
	\path
    (r) edge (r2)
	(r1) edge (r)
	(r1) edge (r2)
	(r1) edge (r3)
	(r1) edge (r4)
	(r2) edge (r3)
	(r2) edge (r5)
	(r3) edge (r4)
	(r3) edge (r5)
	(r6) edge (r4)
	(r5) edge (r7)
	(r5) edge (r6)
	(r4) edge (r7)
	(r6) edge (r7)
    (r6) edge (r8)
	(r7) edge (r8) ;
	\end{tikzpicture}}
\\
\subfloat[$D_1$]{\begin{tikzpicture}[scale=0.9]
		\vertex[fill] (r1) at (0,.8) [] {};
		\vertex[fill] (r2) at (.8,1) [] {};
		\vertex[fill] (r3) at (0,0) [] {};
		\vertex[fill] (r4) at (.8,-.2) [] {};
		\vertex[fill] (r5) at (1.6,.8) [] {};
		\vertex[fill] (r6) at (1.6,0) [] {};
		\vertex[fill] (r7) at (2.3,.4) [] {};
		\path
		(r5) edge (r2)
		(r1) edge (r2)
		(r1) edge (r5)
		(r1) edge (r3)
		(r1) edge (r4)
		(r2) edge (r3)
		(r2) edge (r4)
		(r3) edge (r6)
		(r4) edge (r3)
		(r4) edge (r6)
		(r5) edge (r6)
		(r7) edge (r6)
		(r5) edge (r7);
		\end{tikzpicture}}
	\quad \quad
	\subfloat[$D_2$]{\begin{tikzpicture}[scale=0.9]
		\vertex[fill] (r1) at (0,0) [] {};
		\vertex[fill] (r2) at (0,.8) [] {};
		\vertex[fill] (r3) at (.8,1) [] {};
		\vertex[fill] (r4) at (.8,-.2) [] {};
		\vertex[fill] (r5) at (1.3,.4) [] {};
		\vertex[fill] (r6) at (2,.8) [] {};
		\vertex[fill] (r7) at (2,0) [] {};
		\vertex[fill] (r8) at (2.7,.4) [] {};
		\path
		(r5) edge (r2)
		(r1) edge (r2)
		(r1) edge (r5)
		(r1) edge (r3)
		(r1) edge (r4)
		(r2) edge (r3)
		(r2) edge (r4)
		(r3) edge (r6)
		(r4) edge (r3)
		(r4) edge (r7)
		(r5) edge (r7)
		(r5) edge (r6)
		(r6) edge (r7)
		(r7) edge (r8)
		(r6) edge (r8)  ;
		\end{tikzpicture}}
	\quad
\subfloat[$D_3$]{\begin{tikzpicture}[scale=.8]
	\vertex[fill] (r1) at (0,1) [] {};
	\vertex[fill] (r2) at (.8,1.2) [] {};
	\vertex[fill] (r3) at (0,0) [] {};
	\vertex[fill] (r4) at (.8,-.2) [] {};
	\vertex[fill] (r5) at (1.6,1) [] {};
	\vertex[fill] (r6) at (1.6,0) [] {};
	\vertex[fill] (r7) at (2.6,1) [] {};
	\vertex[fill] (r8) at (2.6,0) [] {};
\vertex[fill] (r9) at (3.3,.5) [] {};
	\path
	(r5) edge (r2)
	(r1) edge (r2)
	(r1) edge (r5)
	(r1) edge (r3)
	(r1) edge (r4)
	(r2) edge (r3)
	(r2) edge (r4)
	(r3) edge (r6)
	(r4) edge (r3)
	(r4) edge (r6)
	(r5) edge (r8)
	(r5) edge (r7)
	(r6) edge (r8)
	(r6) edge (r7)
	(r7) edge (r8)
    (r9) edge (r7)
	(r9) edge (r8)
	;
	\end{tikzpicture}}
\quad
\subfloat[$D_4$]{\begin{tikzpicture}[scale=.9]
	\vertex[fill] (r) at (0,0) [] {};
	\vertex[fill] (r1) at (.5,.5) [] {};
	\vertex[fill] (r2) at (.5,-.5) [] {};
	\vertex[fill] (r3) at (1.5,.5) [] {};
	\vertex[fill] (r4) at (1.5,-.5) [] {};
    \vertex[fill] (r5) at (2,0) [] {};
	\path
	(r) edge (r1)
	(r) edge (r2)
	(r) edge (r3)
	(r) edge (r4)
	(r1) edge (r2)
	(r1) edge (r3)
	(r1) edge (r4)
	(r2) edge (r3)
	(r2) edge (r4)
    (r5) edge (r3)
	(r5) edge (r4)
	;
	\end{tikzpicture}}
\caption{Short blocks of a quartic graph}\label{fig:short}
\end{figure}
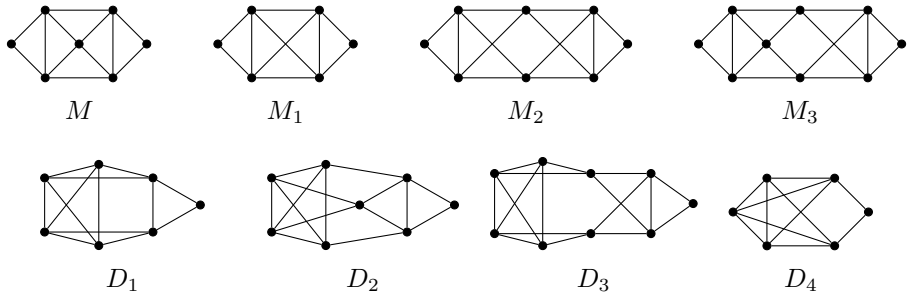

The {\em long} blocks, roughly speaking,  are constructed by putting some short blocks together with the general structure given in Figure~\ref{fig:long}.
\begin{figure}[h!]
\centering
\begin{tikzpicture}[scale=.9]
   \tikzstyle{vertex}=[draw, inner sep=0pt, minimum size=0pt]
 	\vertex[fill] (r17) at (9,0) [label=left:\footnotesize{$$}] {};
 	\vertex[fill] (r) at (1,0) [label=left:\footnotesize{$$}] {};
 	\vertex[fill] (r1) at (1.5,.5) [] {};
 	\vertex[fill] (r2) at (1.5,-.5) [] {};
 	\vertex[fill] (r3) at (2.5,.5) [] {};
 	\vertex[fill] (r4) at (2.5,-.5) [] {};
 	\vertex[fill] (r5) at (3.5,.5) [] {};
 	\vertex[fill] (r6) at (3.5,-.5) [] {};
 	\vertex[fill] (r7) at (4.5,.5) [] {};
 	\vertex[fill] (r8) at (4.5,-.5) [] {};
 	\vertex[fill] (r9) at (5.5,.5) [] {};
 	\vertex[fill] (r10) at (5.5,-.5) [] {};
 	\vertex[fill] (r11) at (6.5,.5) [] {};
 	\vertex[fill] (r12) at (6.5,-.5) [] {};
 	\vertex[fill] (r13) at (7.5,.5) [] {};
 	\vertex[fill] (r14) at (7.5,-.5) [] {};
 	\vertex[fill] (r15) at (8.5,.5) [] {};
 	\vertex[fill] (r16) at (8.5,-.5) [] {};
 	\vertex[fill] (b1) at (1.5,0) [label=right:\footnotesize{$B_1$}] {};
 	\vertex[fill] (b1) at (3.7,0) [label=right:\footnotesize{$B_2$}] {};
 	\vertex[fill] (b1) at (7.7,0) [label=right:\footnotesize{$B_s$}] {};
 	\tikzstyle{vertex}=[ draw, inner sep=0pt, minimum size=0pt]
 	\vertex[] (1) at (5.75,.5) [] {};
 	\vertex[] (2) at (5.95,.5) [] {};
 	\vertex[] (3) at (6.15,.5) [] {};
 	\vertex[] (11) at (5.85,.5) [] {};
 	\vertex[] (22) at (6.05,.5) [] {};
 	\vertex[] (33) at (6.25,.5) [] {};
 	\vertex[] (111) at (5.75,-.5) [] {};
 	\vertex[] (222) at (5.95,-.5) [] {};
 	\vertex[] (333) at (6.15,-.5) [] {};
 	\vertex[] (1111) at (5.85,-.5) [] {};
 	\vertex[] (2222) at (6.05,-.5) [] {};
 	\vertex[] (3333) at (6.25,-.5) [] {};
 	\path
 	(r) edge (r1)
 	(r) edge (r2)
 	(r2) edge (r4)
 	(r1) edge (r3)
 	(r3) edge (r4)
 	(r3) edge (r5)
 	(r4) edge (r6)
 	(r5) edge (r6)
 	(r5) edge (r6)
 	(r7) edge (r8)
 	(r5) edge (r7)
 	(r6) edge (r8)
 	(r7) edge (r9)
 	(r8) edge (r10)
 	(r11) edge (r13)
 	(r12) edge (r14)
 	(r13) edge (r14)
 	(r13) edge (r15)
 	(r14) edge (r16)
 	(r16) edge (r17)
 	(r15) edge (r17)
 	(1) edge (11)
 	(2) edge (22)
 	(3) edge (33)
 	(111) edge (1111)
 	(222) edge (2222)
 	(333) edge (3333);
 	\end{tikzpicture}
 \caption{General structure of a long block}\label{fig:long}
\end{figure}
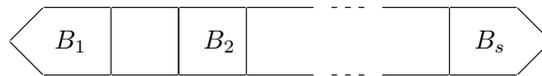

More precisely, the building `bricks' of long blocks  are the graphs $M'_1,M'_2,M'_3,D'_3,D'_4$, obtained by removing the {\em right} degree $2$ vertex of
the corresponding short blocks, and the graphs $M''_1,M''_2$,  obtained by removing both degree $2$ vertices of $M_1,M_2$.
For any of these graphs, say $B$, we denote its mirror image by $\tilde B$.
A long block is constructed from some $s\ge2$ bricks $B_1, \ldots, B_s$, where each $B_i$ is joined by two edges to $B_{i+1}$ (as shown in Figure~\ref{fig:long}).
There are three types of long blocks:
\begin{itemize}
\item[(i)] {\em long end block}: $B_1\in\{D'_3, D'_4\}$,  $B_2, \ldots, B_{s-1}\in \{ M''_1, M''_2\}$, 	and $B_s \in \{\tilde M'_1, \tilde M'_2, \tilde M'_3\}$;
\item[(ii)] {\em  long middle block}:  $B_1 \in\{M'_1, M'_2, M'_3\}$,  $B_2, \ldots, B_{s-1}\in\{M''_1, M''_2\}$, and $B_s \in\{\tilde M'_1, \tilde M'_2, \tilde M'_3\}$;
\item[(iii)] {\em  long complete block}: $B_1\in\{D'_3, D'_4\}$,  $B_2, \ldots, B_{s-1} \in\{M''_1, M''_2\}$, and $B_1\in\{\tilde D'_3, \tilde D'_4\}$
 \end{itemize}

We note that long complete blocks are quartic and long end blocks and middle blocks have exactly one or two vertices of degree $2$, respectively.

Here is the main result of this section.

 \begin{theorem} \label{thm:quartic}
Let $G$ be a graph with the minimum spectral gap in the family of connected quartic graphs on $n$ vertices.
If $G$ is a block, then either $n\le9$ and $G$ is one of the graphs of Figure~\ref{fig:G5-G9}, or $n\ge10$ and  $G$ is a long complete block.
 If  $G$ itself is not a block, then it has a path-like structure in which each left end  block is either one of $D_1,\ldots,D_4$ or a long end block, and
 each middle block is either one of $M,M_1,M_2,M_3,\tilde M_3$ or a long middle block.
  Each right end block is the mirror image of some left end block.
\end{theorem}

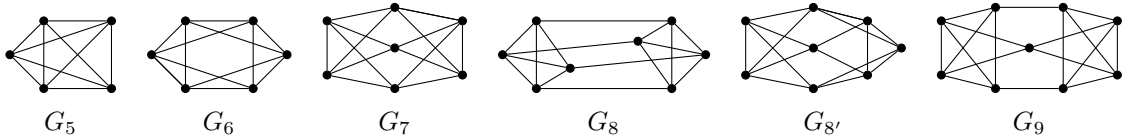
\begin{figure}[h!]
\captionsetup[subfigure]{labelformat=empty}
\centering
\subfloat[$G_5$]{\begin{tikzpicture}[scale=.9]
	\vertex[fill] (1) at (0,0) [] {};
	\vertex[fill] (2) at (.5,.5) [] {};
	\vertex[fill] (3) at (.5,-.5) [] {};
	\vertex[fill] (4) at (1.5,.5) [] {};
	\vertex[fill] (5) at (1.5,-.5) [] {};
	\path
		(1) edge (2)
		(1) edge (3)
	   (1) edge (4)
	   (1) edge (5)
		(2) edge (3)
		(2) edge (4)
        (2) edge (5)
	    (3) edge (4)
		(3) edge (5)
	    (4) edge (5)
		;
\end{tikzpicture}}
\quad
\subfloat[$G_6$]{\begin{tikzpicture}[scale=.9]
	\vertex[fill] (1) at (0,0) [] {};
	\vertex[fill] (2) at (.5,.5) [] {};
	\vertex[fill] (3) at (.5,-.5) [] {};
	\vertex[fill] (4) at (1.5,.5) [] {};
	\vertex[fill] (5) at (1.5,-.5) [] {};
   \vertex[fill] (6) at (2,0) [] {};
	\path
		(1) edge (2)
		(1) edge (3)
	   (1) edge (4)
	   (1) edge (5)
		(2) edge (3)
		(2) edge (4)
        (2) edge (6)
		(3) edge (1)
	    (3) edge (5)
	    (3) edge (6)
		(4) edge (6)
	   (5) edge (6)
		(5) edge (4)
	    ;
\end{tikzpicture}}
\quad
\subfloat[$G_7$]{\begin{tikzpicture}[scale=.9]
    \vertex[fill] (1) at (0,.9) [] {};
	\vertex[fill] (2) at (0,.1) [] {};
	\vertex[fill] (3) at (1,1.1) [] {};
	\vertex[fill] (4) at (1,.5) [] {};
	\vertex[fill] (5) at (1,-.1) [] {};
   \vertex[fill] (6) at (2,.9) [] {};
   \vertex[fill] (7) at (2,.1) [] {};
	\path
		(1) edge (2)
		(1) edge (3)
	   (1) edge (4)
	   (1) edge (5)
		(2) edge (3)
		(2) edge (4)
        (2) edge (5)
		(3) edge (6)
	    (3) edge (7)
	    (3) edge (6)
		(4) edge (6)
	   (4) edge (7)
		(5) edge (6)
	(5) edge (7)
	(6) edge (7)
	;
\end{tikzpicture}}
\quad
\subfloat[$G_8$]{\begin{tikzpicture}[scale=.9]
    \vertex[fill] (1) at (0,0) [] {};
	\vertex[fill] (2) at (.5,.5) [] {};
	\vertex[fill] (3) at (.5,-.5) [] {};
	\vertex[fill] (4) at (1,-.2) [] {};
	\vertex[fill] (5) at (2,.2) [] {};
   \vertex[fill] (6) at (2.5,.5) [] {};
   \vertex[fill] (7) at (2.5,-.5) [] {};
   \vertex[fill] (8) at (3,0) [] {};
	\path
		(1) edge (2)
		(1) edge (3)
	   (1) edge (4)
	   (1) edge (5)
		(2) edge (3)
		(2) edge (4)
        (2) edge (6)
		(3) edge (4)
	    (3) edge (7)
		(4) edge (8)
	   (5) edge (6)
		(5) edge (7)
	(5) edge (8)
	(6) edge (7)
	(6) edge (8)
	(7) edge (8)
	    ;
\end{tikzpicture}}
\quad
\subfloat[$G_{8'}$]{\begin{tikzpicture}[scale=.9]
	 \vertex[fill] (1) at (.2,.8) [] {};
	\vertex[fill] (2) at (.2,0) [] {};
	\vertex[fill] (3) at (1.2,1) [] {};
	\vertex[fill] (4) at (1.2,.4) [] {};
	\vertex[fill] (5) at (1.2,-.2) [] {};
   \vertex[fill] (6) at (2,.8) [] {};
   \vertex[fill] (7) at (2,0) [] {};
   \vertex[fill] (8) at (2.5,.4) [] {};
	\path
		(1) edge (2)
		(1) edge (3)
	   (1) edge (4)
	   (1) edge (5)
		(2) edge (3)
		(2) edge (4)
        (2) edge (5)
		(3) edge (6)
	    (3) edge (8)
	    (3) edge (6)
		(4) edge (6)
	   (4) edge (7)
		(5) edge (8)
	  (5) edge (7)
	 (6) edge (8)
      (6) edge (7)
	 (8) edge (7);
\end{tikzpicture}}
\quad
\subfloat[$G_9$]{\begin{tikzpicture}[scale=.9]
	 \vertex[fill] (1) at (0,.9) [] {};
	\vertex[fill] (2) at (0,.1) [] {};
	\vertex[fill] (3) at (.8,1.1) [] {};
	\vertex[fill] (4) at (.8,-.1) [] {};
	\vertex[fill] (5) at (1.3,.5) [] {};
   \vertex[fill] (6) at (1.8,1.1) [] {};
   \vertex[fill] (7) at (1.8,-.1) [] {};
   \vertex[fill] (8) at (2.6,.9) [] {};
\vertex[fill] (9) at (2.6,.1) [] {};
	\path
		(1) edge (2)
		(1) edge (3)
	   (1) edge (4)
	   (1) edge (5)
		(2) edge (3)
		(2) edge (4)
        (2) edge (5)
	    (3) edge (4)
		(3) edge (6)
	    (4) edge (7)
	     (5) edge (8)
	      (5) edge (9)
	      (6) edge (9)
	      (8) edge (6)
	     (8) edge (7)
	      (6) edge (7)
	      (7) edge (9)
	       (8) edge (9);
\end{tikzpicture}}
\caption{The graphs of Theorem~\ref{thm:quartic} on $n\le9$ vertices}\label{fig:G5-G9}
\end{figure}

Subsection~\ref{subsec:proof} is devoted to the proof of Theorem~\ref{thm:quartic}. In fact, Theorem~\ref{thm:quartic} follows from Theorems~\ref{thm:quarticTOmaximal} and \ref{thm:HisoG} below.

\subsection{Elementary moves and their effect on algebraic connectivity}

In this subsection we present the main tool of the proof of Theorem~\ref{thm:quartic}, that is, a local operation on edges of a graph which preserves the degree sequence of the graph.

  Let $G$ be a graph. By `$\sim$' and `$\nsim$' we denote, respectively, adjacency and non-adjacency in $G$.
An {\em elementary move} or {\em switching}  in $G$ is a switching of parallel edges: let $a \sim b,c \sim d$ and  $a \nsim c,b \nsim d$, then the elementary move denoted by $\sw(a, b, c, d)$ removes the edges $ab$ and $cd$ and replaces
them  by the edges $ac$ and $bd$.
\begin{definition} \label{ro}\rm
Let $G$ be a graph and  $\rho: V(G)\longrightarrow \mathbb{R}$  be a Fiedler vector of $G$, considered as a weighting on the vertices; for $v \in V(G)$ we write $\rho_v=\rho(v)$.
 For convenience, we may assume the vertex set is $[n] =\{1, 2, \ldots, n\}$ and that the vertices are numbered so that $\rho_1\geq \rho_2\geq\cdots\geq\rho_n$.
We call this a {\em proper labeling} of the vertices (with respect to the eigenvector $\rho$).
\end{definition}

The following two lemmas were initially used by Guiduli \cite{Guiduli} (see also \cite{GuiduliThesis}) for cubic graphs but they also hold  for quartic graphs.

\begin{lemma}\label{sw}
Let $G$ be a connected graph.
Let $\rho: V(G)\longrightarrow \mathbb{R}$ be
a Fiedler vector of $G$.  If there are vertices $\{a, b, c, d\}$ in $G$ such that $a \sim b$,
 $c \sim d$,  $a \nsim c$, $b \nsim d$, with $\rho_a\geq\rho_d$, and  $\rho_c\geq\rho_b$,   then  $\sw(a, b, c, d)$ does not increase the algebraic connectivity.
\end{lemma}

\begin{definition} \rm
A switch or elementary move is said to be {\em proper} if it satisfies the conditions of Lemma~\ref{sw}.
In particular, with proper labeling on the vertices, $\sw(a, b, c, d)$ is proper if $a<d$ and  $c<b$.
\end{definition}
We will use proper switchings  to  transfer the graphs  into  the path-like structure without increasing the algebraic connectivity. The following lemma keeps the graph connected during this procedure.
\begin{lemma}\label{lem:connected}
Let $G$ be a properly labeled connected graph on $[n]$ and $n'<n$.
  Assume that $G\setminus [n']$ is disconnected and that each of its
components has an edge which is not a bridge. Then we may reconnect the graph using proper
elementary moves to make $G\setminus [n']$ connected, not increasing the  algebraic connectivity.
\end{lemma}

 In the arguments which follow, we use proper elementary moves to connect two specific vertices $x$ and $y$.
 The following remark demonstrates when such a switch does,  or does not, exist.

 \begin{remark} \label{remark}\rm
 Let $G$ be a graph whose vertices $[n]$ are properly labeled and $x,y$  be two vertices of  $G$ with $x<y$.
 Suppose we are looking for a proper switch to connect $x$ and $y$ without altering the induced subgraph on $[x]$.
 From Lemma~\ref{sw} it is evident that such a switch does not exist if and only if
  any neighbor of $x$ in $[n]\setminus[y]$ is adjacent to any neighbor of $y$ in $[n]\setminus[x]$.
   \end{remark}

\subsection{Proof of Theorem~\ref{thm:quartic}} \label{subsec:proof}

 Theorem~\ref{thm:quartic} follows from Theorems~\ref{thm:quarticTOmaximal} and \ref{thm:HisoG}, which will be proved in this subsection.

Hereafter, we assume that $\Gamma$ is a connected quartic graph with $n$ vertices, whose vertices are labeled properly as described in Definition~\ref{ro}.
Our goal is to utilize proper elementary moves to transfer $\Gamma$ to one of the graphs described in Theorem~\ref{thm:quartic}.

 \subsubsection{The subgraph on  the first few vertices}

  Our first goal  is to prove that we can reconnect (by proper elementary moves) the first few vertices of $\Gamma$ to get one of the four subgraphs  $D_1$, $D_2$, $D'_3$, $D'_4$.

  In the process of reconnecting $\Gamma$, we are usually in the situation that for some $r<n$, we have already built some specific subgraph on  $[r]$  and  continue to build a desired subgraph on  $\Gamma\setminus[r]$ in a way not to alter the subgraph already constructed on $[r]$.   The next two lemmas deal with such situations.

\begin{lemma}\label{lem:r+1--r+2}
	Let $H$ be a connected graph with vertices $r+1,\ldots,r+m$ where all the vertices have degree $4$ except the first  two, which
have degree $3$.   Then, by proper  switchings, $H$ can be transferred into a graph in which $r+1\sim r+2$, or  $m=5,8$ and $H$ can be transferred into $\tilde D'_4$ or $\tilde D'_3$, respectively.
\end{lemma}
\begin{proof}
If some neighbor $x$ of  $r + 1$ is not adjacent to some neighbor $y$ of  $r + 2$, then  $\sw(r + 1, x, r + 2, y)$ connects $r+1$ to $r+2$.
Otherwise any neighbor of $r+1$ is adjacent with any neighbor of $r+2$.
This is only possible in two cases: (i) $r + 1$ and $r + 2$  share three neighbors  all of which  are adjacent  to each other or  (ii) $H$ is the graph of Figure~\ref{fig:lemma3.7a}.  If (i) is the case, then $m =  5$ and $H$ is  $\tilde D'_4$.
 In the case (ii), with no loss of generality, we assume that $x<z<k$, $y<w<l$, and $x<y$.
We first $\sw(r+2, y, x, w)$ and then $\sw(r+2, x, r + 1, z)$, which  result in Figure~\ref{fig:lemma3.7b}.
Now we perform $\sw(r+2, w, x, l)$, and then either $\sw(r+2, l ,k, y)$ if $k<l$ or $\sw(r+1, k, l, z)$ if $l<k$.
As a result we obtain  $\tilde D'_3$ (the outcome  in the case $k<l$ is shown in Figure~\ref{fig:lemma3.7c}).
\end{proof}

 \begin{figure}[h!]
\centering
\subfloat[]{\begin{tikzpicture}[scale=.95]
\vertex[fill] (r1) at (2,1) [label=above:\scriptsize{$r+1$}] {};
\vertex[fill] (r2) at (2,-1) [label=below:\scriptsize{$r+2$}] {};
\vertex[fill] (x) at (1.2,.4) [label=left:\scriptsize{$x$}] {};
\vertex[fill] (z) at (2,.4) [label=right:\scriptsize{$z$}] {};
\vertex[fill] (k) at (2.8,.4) [label=right:\scriptsize{$k$}] {};
\vertex[fill] (y) at (1.2,-.4) [label=left:\scriptsize{$y$}] {};
\vertex[fill] (w) at (2,-.4) [label=right:\scriptsize{$w$}] {};
\vertex[fill] (l) at (2.8,-.4) [label=right:\scriptsize{$l$}] {};
\tikzstyle{vertex}=[circle, draw, inner sep=0pt, minimum size=0pt]
    \vertex (s) at (1.6,.9) [label=right:$$] {};
	\vertex (ss) at (1.6,-.9) [label=right:$$] {};
	\path
	    (r1) edge (x)
		(r1) edge (z)
	    (r1) edge (k)
	    (r2) edge (y)
		(r2) edge (w)
	    (r2) edge (l)
	     (x) edge (y)
		(x) edge (w)
	    (x) edge (l)
	     (z) edge (y)
		(z) edge (w)
	    (z) edge (l)
	     (k) edge (y)
		(k) edge (w)
	    (k) edge (l);
\end{tikzpicture}\label{fig:lemma3.7a}}
 \quad \quad \quad
\subfloat[]{\begin{tikzpicture}
	\vertex[fill] (r1) at (2,1) [label=above:\scriptsize{$r+1$}] {};
    \vertex[fill] (r2) at (2,-1) [label=below:\scriptsize{$r+2$}] {};
	\vertex[fill] (x) at (1.2,.4) [label=left:\scriptsize{$x$}] {};
\vertex[fill] (z) at (2,.4) [label=right:\scriptsize{$z$}] {};
\vertex[fill] (k) at (2.8,.4) [label=right:\scriptsize{$k$}] {};
\vertex[fill] (y) at (1.2,-.4) [label=left:\scriptsize{$y$}] {};
\vertex[fill] (w) at (2,-.4) [label=right:\scriptsize{$w$}] {};
\vertex[fill] (l) at (2.8,-.4) [label=right:\scriptsize{$l$}] {};
\tikzstyle{vertex}=[circle, draw, inner sep=0pt, minimum size=0pt]
    \vertex (s) at (1.6,.9) [label=right:$$] {};
	\vertex (ss) at (1.6,-.9) [label=right:$$] {};
	\path
      (r1) edge[bend right=30] (r2)
	    (r1) edge (x)
	    (r1) edge (k)
		(r2) edge (w)
	    (r2) edge (l)
	     (x) edge (y)
	    (x) edge (l)
	     (z) edge (y)
		(z) edge (w)
	    (z) edge (l)
	     (k) edge (y)
		(k) edge (w)
	    (k) edge (l)
	    (y) edge (w)
	    (x) edge (z) ;
\end{tikzpicture}\label{fig:lemma3.7b}}
\quad \quad \quad
\subfloat[] {\begin{tikzpicture}[scale=.95]
	\vertex[fill] (r1) at (0,1) [label=above:\scriptsize{$z$}] {};
	\vertex[fill] (r2) at (-.9,1.2) [label=above:\scriptsize{$y$}] {};
	\vertex[fill] (r3) at (0,0) [label=below:\scriptsize{$l$}] {};
	\vertex[fill] (r4) at (-.9,-.2) [label=below:\scriptsize{$w$}] {};
    \vertex[fill] (r5) at (-1.8,1) [label=above:\scriptsize{$x$}] {};
    \vertex[fill] (r6) at (-1.8,0) [label=below:\scriptsize{$k$}] {};
     \vertex[fill] (r7) at (-2.8,1) [label=above:\scriptsize{$r+1$}] {};
    \vertex[fill] (r8) at (-2.8,0) [label=below:\scriptsize{$r+2$}] {};
   	\path
		(r5) edge (r2)
		(r1) edge (r2)
	    (r1) edge (r5)
		(r1) edge (r3)
        (r1) edge (r4)
		(r2) edge (r3)
	    (r2) edge (r4)
	    (r3) edge (r6)
	    (r4) edge (r3)
	    (r4) edge (r6)
	    (r5) edge (r8)
	    (r5) edge (r7)
	    (r6) edge (r8)
	    (r6) edge (r7)
	    (r7) edge (r8)
	   	   ;
\end{tikzpicture}\label{fig:lemma3.7c}}
 \caption{The graph $H$ with $m=8$ before and after switchings  in the proof of Lemma~\ref{lem:r+1--r+2}}\label{fig:lemma3.7}
\end{figure}
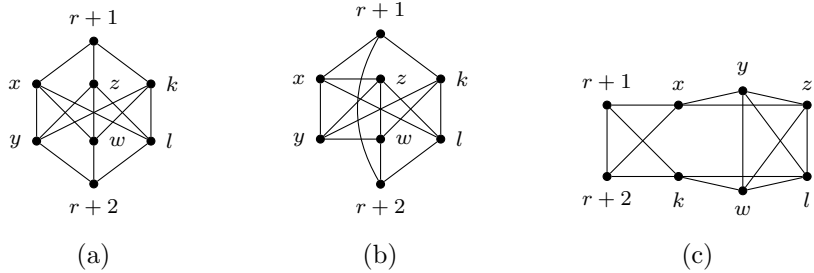

  \begin{lemma}\label{lem:r+1,r+2VerCut}
  	Let $H$ be a graph with vertices $r+1,\ldots,r+m$ where all the vertices have degree $4$ except for $r+1,r+2$ which have exactly two neighbors among  $r+3,\ldots,r+m$. Further $r+1\sim r+3,r+4$.  Then by proper switchings, $H$ can be transferred into a graph with $r+2\sim r+3$, or  $m=6$ and $H$ can
  	be transferred into the the graph $\bar D_1$ (depicted in Figure~\ref{fig:lemma3.8b}).
\end{lemma}
\begin{proof}  First suppose that $r+2 \sim r +4$. If  $r+2\nsim  r+3$, then $r+3$ has a neighbor $x$ which is not adjacent with $r+4$ (otherwise $r+4$ should have degree $5$). Then  $\sw(r+2, r +4, r+3,x)$ makes $r+2$ and $r+3$ adjacent.

Next suppose that $r+2 \nsim r +4$.
If some neighbor $x$ of  $r + 3$ is not adjacent to some neighbor $y$ of  $r + 4$, then  $\sw(r + 3, x, r + 4, y)$ connects $r+3$ to $r+4$. Otherwise  $r + 3$ and $r + 4$ must share three neighbors  all of which  are adjacent  to each other.
In this case, the neighbors of $r+2$ and the neighbors of $r+3$ cannot be adjacent. So there is a switch to connect $r+2$ to $r+3$ and we are done.
  Therefore we suppose that   $r+3\sim  r+4$.
If some neighbor of $r+2$ is not adjacent with some neighbor of $r+3$, then we are done. Otherwise either
 we are in the situation of  Figure~\ref{fig:lemma3.8a} and so by $\sw(r+2,y,r+4,x)$,  $r+2\sim r+4$, which implies $r+2\sim r+3$ as discussed above; or $r+2$ and $r+3$ share two neighbors, say {\w $x,y$,} such that $x \sim y$ and $r+4$ is adjacent to both  $x$ and $y$, which means  that $m=6$ and $H$ is the graph $\bar D_1$ given in Figure~\ref{fig:lemma3.8b}. Note  that $r+1$ and $r+2$ can be either adjacent or non-adjacent, which is illustrated by dashed edges in Figure~\ref{fig:lemma3.8}.
\end{proof}
\begin{figure}[h!]
\centering
\subfloat[]{\begin{tikzpicture}[scale=.9]
	\vertex[fill] (r1) at (1.5,.2) [label=left:\scriptsize{$r+1$}] {};
	\vertex[fill] (r2) at (1.5,-1) [label=left:\scriptsize{$r+2 \  $}] {};
	\vertex[fill] (r3) at (3.3,.2) [label=right:\scriptsize{$r+3$}] {};
	\vertex[fill] (r4) at (3.3,-1) [label=right:\scriptsize{$r+4$}] {};
	\vertex[fill] (x) at (2,-1.3) [label=above:\scriptsize{$x$}] {};
	\vertex[fill] (v) at (2.6,-1.3) [label=left:$$] {};
	\vertex[fill] (y) at (1.8,-2) [label=left:\scriptsize{$y$}] {};
	\vertex[fill] (z) at (2.8,-2) [label=right:$$] {};
	\path
		(r1) edge (r3)
		(r3) edge (r4)
	    (r2) edge (x)
		(r2) edge (y)
		(x) edge (v)
		(x) edge (z)
		(y) edge (v)
	    (y) edge (z)
		(y) edge (r4)
		(r3) edge (v)
		(r3) edge (z)
	    (r4) edge (x)
	    (r1) edge (r4);
	     \path[dashed]
	       (r1) edge (r2);
\end{tikzpicture}\label{fig:lemma3.8a}}
\quad \quad \quad
\subfloat[]{\begin{tikzpicture}
    \vertex[fill] (r7) at (3.5,.5) [label=left:\scriptsize{$r+1$}] {};
	\vertex[fill] (r8) at (3.5,-.5) [label=left:\scriptsize{$r+2$}] {};
    \vertex[fill] (r9) at (4.3,.7) [label=above:\scriptsize{$r+4$}] {};
	\vertex[fill] (r10) at (4.3,-.7) [label=below:\scriptsize{$x$}] {};
    \vertex[fill] (r11) at (5.3,.5) [label=right:\scriptsize{$r+3$}] {};
	\vertex[fill] (r12) at (5.3,-.5) [label=right:\scriptsize{$y$}] {};
	\path
	    (r12) edge (r8)
	    (r10) edge (r8)
        (r9) edge (r7)
	    (r11) edge (r7)
	    (r9) edge (r10)
	    (r9) edge (r11)
	    (r9) edge (r12)
	    (r10) edge (r11)
	    (r10) edge (r12)
	    (r11) edge (r12);
	     \path[dashed]
	       (r7) edge (r8);
\end{tikzpicture}\label{fig:lemma3.8b}}
 \caption{An exceptional case  in the proof of Lemma~\ref{lem:r+1,r+2VerCut} and the graph $\bar D_1$}\label{fig:lemma3.8}
\end{figure}
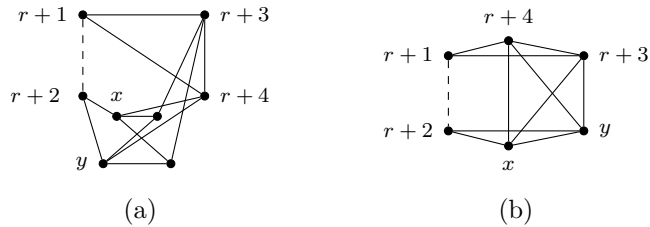

Now, we can prove that the first few vertices of $\Gamma$ can be reconnected to induce the subgraphs asserted in Theorem \ref{thm:quartic}.
 \begin{lemma}\label{lem:begin1}
 The induced subgraphs on the first few vertices in $\Gamma$ can be  transferred by proper switchings into one of the four subgraphs $D_1$, $D_2$, $D'_3$, $D'_4$. Furthermore, if $n\leq 9$, then $\Gamma$ can be  transferred  into one of the graphs $G_5, G_6, G_7, G_8, G_{8'}$, or  $G_9$.
  \end{lemma}
 \begin{proof}
 In Steps 1--5 below, we show that the induced subgraph on the first five to seven vertices of $\Gamma$ can be  transferred  into $D'_4$ or to one of the subgraphs $H_1,H_2$ given in Figure~\ref{fig:H1H2}, or $\Gamma$ has at most 9 vertices and it is one of the graphs $G_5, G_6, G_7, G_8, G_{8'}$, or $G_9$.
  In the final Step 6, from $H_1,H_2$ we obtain one of $D_1, D_2, D'_3, D'_4$.
\begin{figure}[h!]
\centering\captionsetup[subfigure]{labelformat=empty}
\subfloat[$H_1$]{\begin{tikzpicture}
	\vertex[fill] (r1) at (0,.8) [label=above:\scriptsize{$1$}] {};
	\vertex[fill] (r2) at (.8,1) [label=above:\scriptsize{$2$}] {};
	\vertex[fill] (r3) at (0,0) [label=below:\scriptsize{$3$}] {};
	\vertex[fill] (r4) at (.8,-.2) [label=below:\scriptsize{$4$}] {};
    \vertex[fill] (r5) at (1.6,.8) [label=above:\scriptsize{$5$}] {};
    \vertex[fill] (r6) at (1.6,0) [label=below:\scriptsize{$6$}] {};
   \tikzstyle{vertex}=[circle, draw, inner sep=0pt, minimum size=0pt]
  	\path
		(r5) edge (r2)
		(r1) edge (r2)
	    (r1) edge (r5)
		(r1) edge (r3)
        (r1) edge (r4)
		(r2) edge (r3)
	    (r2) edge (r4)
	    (r3) edge (r6)
	    (r4) edge (r3)
	    (r4) edge (r6)	  ;
\end{tikzpicture}}
\quad \quad \quad \quad
\subfloat[$H_2$]{\begin{tikzpicture}
	\vertex[fill] (r1) at (0,0) [label=below:\scriptsize{$1$}] {};
	\vertex[fill] (r2) at (0,.8) [label=above:\scriptsize{$2$}] {};
	\vertex[fill] (r3) at (.8,1) [label=above:\scriptsize{$3$}] {};
	\vertex[fill] (r4) at (.8,-.2) [label=below:\scriptsize{$4$}] {};
    \vertex[fill] (r5) at (1.3,.4) [label=right:\scriptsize{$\  5$}] {};
    \vertex[fill] (r7) at (2,.8) [label=above:\scriptsize{$6$}] {};
    \vertex[fill] (r6) at (2,0) [label=below:\scriptsize{$7$}] {};
	\path
		(r5) edge (r2)
		(r1) edge (r2)
	    (r1) edge (r5)
		(r1) edge (r3)
        (r1) edge (r4)
		(r2) edge (r3)
	    (r2) edge (r4)
	    (r4) edge (r3)
	    (r3) edge (r7)
	    (r5) edge (r7)
	    (r5) edge (r6)
	    (r4) edge (r6)   ;
\end{tikzpicture}}
\caption{Two subgraphs on the first few vertices of $\Gamma$}\label{fig:H1H2}
\end{figure}
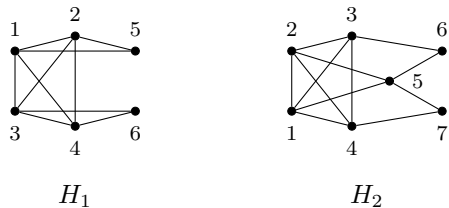

 \noindent \textbf{Step 1.} {\em Connecting  $1$ to $2, 3, 4, 5$ and $2$ to $3$}.
 Assume that $1$ is not adjacent with some $x\in\{2, 3, 4, 5\}$. So $1$ has some neighbor $y>5$.
  It is not possible that $y$ is adjacent with any neighbor of $x$ (this requires $y$ having degree $5$).
  So there is some vertex $z$ such that $z\sim x$ and $z\nsim y$. Now, $\sw(1, y, x, z)$ makes 1 adjacent to $x$.

We may assume that $\Gamma-1$ is connected. The desired switch to $2\sim3$ is possible, unless  $2$ and $3$ share three neighbors in $\Gamma-1$ and all the three neighbors are adjacent to each other. But this contradicts the fact that $\Gamma-1$ is connected.

\noindent \textbf{Step 2.} {\em Connecting $2$ to $4$}. Again we may assume by Lemma~\ref{lem:connected} that
 $\Gamma\setminus \left[3\right] $  is connected.
If no proper switch  leading to $2\sim4$ exists, then,   similarly  to Step 1, we see that $4 \sim 5$.
{\w Also,} we may assume that  $3\nsim4$, because otherwise $2$ has a neighbor $y$ with $y\nsim3$, and so $\sw(2,y,4,3)$ connects $2$ to $4$.
 Let $x$ be a neighbor of $2$ other than  $1$ and $3$. First suppose that $2 \nsim 5$. If $x\sim 4$, then by Remark~\ref{remark} the desired switch  exists, except in the two situations (a) and (b) of Figure~\ref{fig:lemma3.9.4}.
  For (a), we are done by $\sw(2, x, 4, 5)$. Note that (b) is not possible in view of the fact that $\Gamma\setminus \left[3\right] $ is connected.
 If $x\nsim4$, then the desired switch  exists, except in the situation (c) of Figure~\ref{fig:lemma3.9.4} for which $\sw(2, x, 5, y)$ implies $2\sim5$. So we are left with the case that $2 \sim 5$. If $x\nsim 5$, then $\sw(2, x, 4, 5)$. Otherwise, $1,2,4,x$ are all the neighbors of $5$.
 Let $y\neq 1, 5, x$  be the fourth neighbor of $4$.  Then $\sw(2, 5, 4, y)$.

\begin{figure}[h!]
\centering
\subfloat[]{\begin{tikzpicture}[scale=.9]
	\vertex[fill] (1) at (.6,-.4) [label=left:\scriptsize{$1$}] {};
	\vertex[fill] (2) at (1.5,.2) [label=above:\scriptsize{$2 \  \   $}] {};
	\vertex[fill] (3) at (1.5,-1) [label=below:\scriptsize{$3  $}] {};
	\vertex[fill] (4) at (3.3,.2) [label=above:\scriptsize{$\  \  4$}] {};
	\vertex[fill] (5) at (3.3,-1) [label=below:\scriptsize{$5$}] {};
	\vertex[fill] (x) at (2.4,.4) [label=above:\scriptsize{$x$}] {};
	\vertex[fill] (y) at (1.8,.8)[] {};
	\vertex[fill] (z) at (3.1,.8) [] {};
	\path
		(1) edge (2)
		(1) edge (3)
		(1) edge (4)
		(1) edge (5)
	    (2) edge (x)
		(2) edge (y)
	    (x) edge (y)
		(x) edge (z)
		(x) edge (4)
		(4) edge (5)
	    (5) edge (y)
	    (2) edge (3)
	    (4) edge (z)
	    (z) edge (y)
	    ;
\end{tikzpicture}\label{fig:lemma3.9.4a}}
 \quad \quad \quad
\subfloat[]{\begin{tikzpicture}[scale=.9]
	\vertex[fill] (1) at (.6,-.4) [label=left:\scriptsize{$1$}] {};
	\vertex[fill] (2) at (1.5,.2) [label=above:\scriptsize{$2$}] {};
	\vertex[fill] (3) at (1.5,-1) [label=below:\scriptsize{$3  $}] {};
	\vertex[fill] (4) at (3.3,.1) [label=above:\scriptsize{$4$}] {};
	\vertex[fill] (5) at (3.3,-.9) [label=below:\scriptsize{$5$}] {};
	\vertex[fill] (x) at (2.3,.4) [label=below:\scriptsize{$x$}] {};
	\vertex[fill] (z) at (2.3,.9) [] {};
	\path
		(1) edge (2)
		(1) edge (3)
		(1) edge (4)
		(1) edge (5)
	    (5) edge (z)
	    (2) edge (3)
	    (2) edge (x)
		(2) edge (z)
	    (x) edge (z)
		(x) edge (5)
		(x) edge (4)
		 (4) edge (z)
		(4) edge (5);
\end{tikzpicture}\label{fig:lemma3.9.4b}}
\quad \quad \quad
\subfloat[]
{\begin{tikzpicture}[scale=.9]
	\vertex[fill] (1) at (.6,-.4) [label=left:\scriptsize{$1$}] {};
	\vertex[fill] (2) at (1.5,.2) [label=above:\scriptsize{$2$}] {};
	\vertex[fill] (3) at (1.5,-1) [label=below:\scriptsize{$3  $}] {};
	\vertex[fill] (4) at (3.3,.1) [label=above:\scriptsize{$4$}] {};
	\vertex[fill] (5) at (3.3,-.9) [label=below:\scriptsize{$5$}] {};
	\vertex[fill] (x) at (2,.4) [label=below:\scriptsize{$x$}] {};
	\vertex[fill] (y) at (2.1,.9) [label=above:\scriptsize{$y$}] {};
    \vertex[fill] (u) at (2.9,.4) [] {};
	\vertex[fill] (v) at (2.8,.9) [] {};
	\path
		(1) edge (2)
		(1) edge (3)
		(1) edge (4)
		(1) edge (5)
	    (2) edge (3)
	    (2) edge (x)
		(2) edge (y)
	    (x) edge (u)
		(x) edge (v)
		(x) edge (5)
		 (y) edge (u)
		(y) edge (v)
		(y) edge (5)
		 (4) edge (u)
		 (4) edge (v)
		(4) edge (5);
\end{tikzpicture}\label{fig:lemma3.9.4c}}
 \caption{Some possible situations in Step 2}\label{fig:lemma3.9.4}
\end{figure}
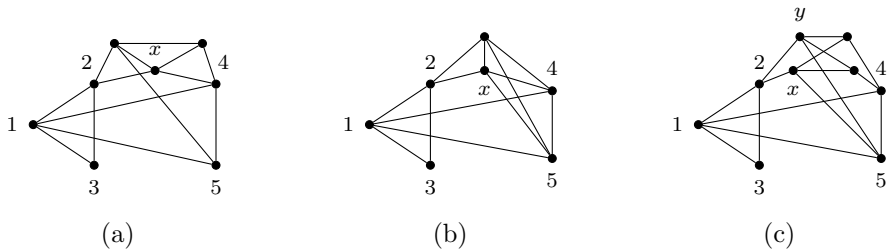

\noindent\textbf{Step 3.} {\em Connecting $2$ to $5$}. 
Let $x\neq 1, 3, 4$  be the fourth neighbor of $2$. We consider the following two cases. 
 \begin{itemize}
  \item[(i)] $x\sim 5$. Let $y$ and $z$  be the other two neighbors of $x$.
  If $y\not\sim5$ or $z\not\sim5$, then the desired switch is possible. Otherwise we are in the situation of  Figure~\ref{fig:lemma3.9.5ia}.
We first show that $3\sim 5$ or $4\sim 5$.
If $y\nsim z$,  then by examining the neighbors of $3$ and $4$, proper switches to  $3\sim 5$ or $4\sim 5$  will clearly exist. If $y\sim z$, then the desired switch  will exist, except when $3\sim 4$, $3\sim y$, and $4\sim z$ in which case $n=8$ and $\Gamma=G_8$.
  So we have that $3\sim5$ or $4\sim5$ and thus we are in either of the situations (b) or (c) of Figure~\ref{fig:lemma3.9.5i}.
(Note that if there is no edge $4x$  in (b) or $3x$ in  (c), then it is easy to find a switch that connects $2$ to $5$.)
For (b), $3$ has a neighbor $y\neq 4$ and $y\nsim x$. Then $\sw(3, y, 5, x)$. For  (c),  $4$ has a neighbor $y\neq 3$ and $y\nsim x$. Then $\sw(4, y, 5, x)$.
It turns out that both (b) and (c) lead to the subgraph (d)
of Figure~\ref{fig:lemma3.9.5i}.
If both $3$ and $4$ are adjacent to $x$,  then $n=6$ and we get $G_6$. Therefore we suppose that both $3$ and $4$ cannot  be adjacent to $x$.
Then either $3\nsim x$ or $4\nsim x$, for which we  apply $\sw(2, x, 5, 3)$
or  $\sw(2, x, 5, 4)$, respectively.

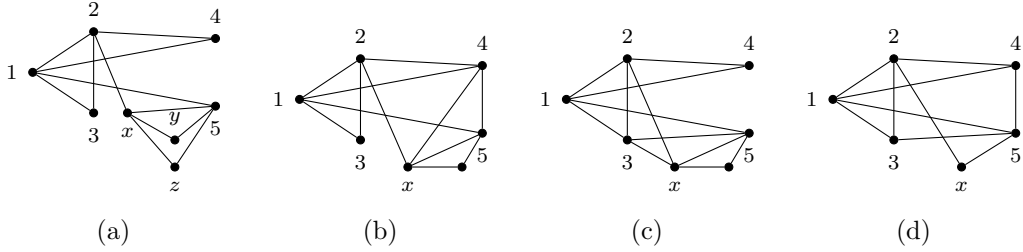
\begin{figure}[h!]
\centering
\subfloat[]{\begin{tikzpicture}[scale=.9]
	\vertex[fill] (1) at (.6,-.4) [label=left:\scriptsize{$1$}] {};
	\vertex[fill] (2) at (1.5,.2) [label=above:\scriptsize{$2$}] {};
	\vertex[fill] (3) at (1.5,-1) [label=below:\scriptsize{$3  $}] {};
	\vertex[fill] (4) at (3.3,.1) [label=above:\scriptsize{$4$}] {};
	\vertex[fill] (5) at (3.3,-.9) [label=below:\scriptsize{$5$}] {};
	\vertex[fill] (x) at (2,-1) [label=below:\scriptsize{$x$}] {};
	\vertex[fill] (y) at (2.7,-1.4) [label=above:\scriptsize{$y$}] {};
\vertex[fill] (u) at (2.7,-1.8) [label=below:\scriptsize{$z$}] {};
	\path
		(1) edge (2)
		(1) edge (3)
		(1) edge (4)
		(1) edge (5)
	    (2) edge (3)
	    (2) edge (x)
	    (2) edge (4)
		(x) edge (y)
	    (x) edge (u)
		(x) edge (5)
		(y) edge (5)
		 (5) edge (u)	;
\end{tikzpicture}\label{fig:lemma3.9.5ia}}
\quad
\subfloat[]{\begin{tikzpicture}[scale=.9]
	\vertex[fill] (1) at (.6,-.4) [label=left:\scriptsize{$1$}] {};
	\vertex[fill] (2) at (1.5,.2) [label=above:\scriptsize{$2$}] {};
	\vertex[fill] (3) at (1.5,-1) [label=below:\scriptsize{$3  $}] {};
	\vertex[fill] (4) at (3.3,.1) [label=above:\scriptsize{$4$}] {};
	\vertex[fill] (5) at (3.3,-.9) [label=below:\scriptsize{$5$}] {};
	\vertex[fill] (x) at (2.2,-1.4) [label=below:\scriptsize{$x$}] {};
	\vertex[fill] (y) at (3,-1.4) [] {};
	\path
		(1) edge (2)
		(1) edge (3)
		(1) edge (4)
		(1) edge (5)
	    (2) edge (3)
	    (2) edge (x)
	    (2) edge (4)
		(x) edge (y)
		(x) edge (5)
		(y) edge (5)
	   (4) edge (5)
	   (4) edge (x);
\end{tikzpicture}\label{fig:lemma3.9.5ib}}
\quad
\subfloat[]{\begin{tikzpicture}[scale=.9]
	\vertex[fill] (1) at (.6,-.4) [label=left:\scriptsize{$1$}] {};
	\vertex[fill] (2) at (1.5,.2) [label=above:\scriptsize{$2$}] {};
	\vertex[fill] (3) at (1.5,-1) [label=below:\scriptsize{$3  $}] {};
	\vertex[fill] (4) at (3.3,.1) [label=above:\scriptsize{$4$}] {};
	\vertex[fill] (5) at (3.3,-.9) [label=below:\scriptsize{$5$}] {};
	\vertex[fill] (x) at (2.2,-1.4) [label=below:\scriptsize{$x$}] {};
	\vertex[fill] (y) at (3,-1.4) [] {};
	\path
		(1) edge (2)
		(1) edge (3)
		(1) edge (4)
		(1) edge (5)
	    (2) edge (3)
	    (2) edge (x)
	    (2) edge (4)
		(x) edge (y)
		(x) edge (5)
		(y) edge (5)
	   (3) edge (5)
	   (3) edge (x);
\end{tikzpicture}\label{fig:lemma3.9.5ic}}
\quad
\subfloat[]{\begin{tikzpicture}[scale=.9]
	\vertex[fill] (1) at (.6,-.4) [label=left:\scriptsize{$1$}] {};
	\vertex[fill] (2) at (1.5,.2) [label=above:\scriptsize{$2$}] {};
	\vertex[fill] (3) at (1.5,-1) [label=below:\scriptsize{$3  $}] {};
	\vertex[fill] (4) at (3.3,.1) [label=above:\scriptsize{$4$}] {};
	\vertex[fill] (5) at (3.3,-.9) [label=below:\scriptsize{$5$}] {};
	\vertex[fill] (x) at (2.5,-1.4) [label=below:\scriptsize{$x$}] {};
	\path
		(1) edge (2)
		(1) edge (3)
		(1) edge (4)
		(1) edge (5)
	    (2) edge (3)
	    (2) edge (x)
	    (2) edge (4)
		(x) edge (5)
	   (4) edge (5)
	   (3) edge (5);
\end{tikzpicture}\label{fig:lemma3.9.5id}}
 \caption{Some possible situations in Step 3, Case (i)}\label{fig:lemma3.9.5i}
\end{figure}

\item[(ii)] $x\nsim5$.
If the remaining neighbors of $x$ and $5$ are not the same, then the desired switch is available.
Otherwise,  similarly to (i), we have $3\sim5$ or $4\sim5$. So in view of Remark~\ref{remark}, we are in either of the situations (a) or (b) of Figure~\ref{fig:lemma3.9.5ii}.
For (a), first let $y\nsim z$. If $x< z$, then $\sw(x,y,5,z)$, and if $z<x$, then $\sw(z,5,4,x)$  connects  $x$ to $5$, which reduces the graph to Case (i). Now let $y\sim z$. If $3\sim y$ and $3\sim z$, then $n=8$, and by $\sw(3,y,5,4)$ and then $\sw(3,5,4,x)$ we transfer $\Gamma$ to $G_8$. If $3\nsim y$ or $3\nsim z$, then there is a neighbor $w$ of $3$ such that either $w\nsim y$ and $w\neq y$, and then $\sw(3,w,5,y)$, or $w\nsim z$ and $w\neq z$. Then $\sw(3,w,5,z)$ connects $3$ to $5$.
We do the same for (b) to connect $4$ to $5$. So both (a) and (b) lead to the subgraph (c) of Figure~\ref{fig:lemma3.9.5ii}.
Now $\sw(3, x,4, 5)$  connects $x$ to $5$, which reduces the graph to Case (i).
\end{itemize}

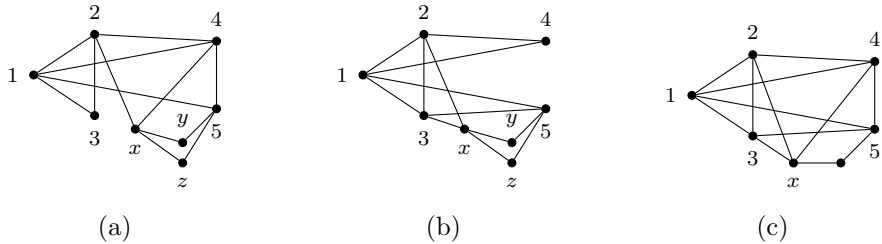
\begin{figure}[h!]
\centering
\subfloat[]{\begin{tikzpicture}[scale=.9]
	\vertex[fill] (1) at (.6,-.4) [label=left:\scriptsize{$1$}] {};
	\vertex[fill] (2) at (1.5,.2) [label=above:\scriptsize{$2$}] {};
	\vertex[fill] (3) at (1.5,-1) [label=below:\scriptsize{$3  $}] {};
	\vertex[fill] (4) at (3.3,.1) [label=above:\scriptsize{$4$}] {};
	\vertex[fill] (5) at (3.3,-.9) [label=below:\scriptsize{$5$}] {};
	\vertex[fill] (x) at (2.1,-1.2) [label=below:\scriptsize{$x$}] {};
\vertex[fill] (y) at (2.8,-1.4) [label=above:\scriptsize{$y$}] {};
	\vertex[fill] (z) at (2.8,-1.7) [label=below:\scriptsize{$z$}] {};
	\path
		(1) edge (2)
		(1) edge (3)
		(1) edge (4)
		(1) edge (5)
	    (2) edge (3)
	    (2) edge (x)
	    (2) edge (4)
		(x) edge (y)
		(y) edge (5)
	   (z) edge (5)
	   (z) edge (x)
	   (4) edge (5)
	   (4) edge (x);
\end{tikzpicture}\label{fig:lemma3.9.5iia}}
\quad \quad \quad
\subfloat[]{\begin{tikzpicture}[scale=.9]
		\vertex[fill] (1) at (.6,-.4) [label=left:\scriptsize{$1$}] {};
	\vertex[fill] (2) at (1.5,.2) [label=above:\scriptsize{$2$}] {};
	\vertex[fill] (3) at (1.5,-1) [label=below:\scriptsize{$3  $}] {};
	\vertex[fill] (4) at (3.3,.1) [label=above:\scriptsize{$4$}] {};
	\vertex[fill] (5) at (3.3,-.9) [label=below:\scriptsize{$5$}] {};
	\vertex[fill] (x) at (2.1,-1.2) [label=below:\scriptsize{$x$}] {};
\vertex[fill] (y) at (2.8,-1.4) [label=above:\scriptsize{$y$}] {};
	\vertex[fill] (z) at (2.8,-1.7) [label=below:\scriptsize{$z$}] {};
	\path
		(1) edge (2)
		(1) edge (3)
		(1) edge (4)
		(1) edge (5)
	    (2) edge (3)
	    (2) edge (x)
	    (2) edge (4)
		(x) edge (y)		
		(y) edge (5)
	   (z) edge (5)
	   (z) edge (x)
	   (3) edge (5)
	   (3) edge (x);
\end{tikzpicture}\label{fig:lemma3.9.5iib}}
\quad \quad \quad
\subfloat[]{\begin{tikzpicture}[scale=.9]
	\vertex[fill] (1) at (.6,-.4) [label=left:\scriptsize{$1$}] {};
	\vertex[fill] (2) at (1.5,.2) [label=above:\scriptsize{$2$}] {};
	\vertex[fill] (3) at (1.5,-1) [label=below:\scriptsize{$3  $}] {};
	\vertex[fill] (4) at (3.3,.1) [label=above:\scriptsize{$4$}] {};
	\vertex[fill] (5) at (3.3,-.9) [label=below:\scriptsize{$5$}] {};
	\vertex[fill] (x) at (2.1,-1.4) [label=below:\scriptsize{$x$}] {};
	\vertex[fill] (y) at (2.8,-1.4) [] {};
	\path
		(1) edge (2)
		(1) edge (3)
		(1) edge (4)
		(1) edge (5)
	    (2) edge (3)
	    (2) edge (x)
	    (2) edge (4)
		(x) edge (y)
		(3) edge (5)
		(3) edge (x)
		(y) edge (5)
	   (4) edge (5)
	   (4) edge (x);
\end{tikzpicture}\label{fig:lemma3.9.5iic}}
 \caption{Some possible situations in Step 3, Case (ii)}\label{fig:lemma3.9.5ii}
\end{figure}

\noindent  \textbf{Step 4.} {\em Connecting $3$ to $4$}.
Let $x\neq 1, 2$  be a neighbor of $3$. If $4\sim 5$, we may choose $x$ so that $x\nsim 5$, and then $\sw(3, x,4, 5)$.
 So assume that $4\nsim5$.  From Remark~\ref{remark},  it is seen that the desired switch is available, except in the situations of Figure~\ref{fig:lemma3.9.6}.
 For each of them, we first show that $4\sim 5$. Then, with this edge, the desired switches can be found. 
In the one in Figure~\ref{fig:lemma3.9.6a}, if $5$   is  adjacent to both  $y$ and $z$, then $n=8$ and $\Gamma=G_{8^\prime}$. Otherwise $5$ has  a neighbor $w\neq y,z$ and $w\nsim x$. We first $\sw(4, x,5, w)$ and then $\sw(3, x,4, 5)$.
The other two situations are similar. Note that in Figure~\ref{fig:lemma3.9.6b}, if $5\sim x$ and $5\sim y$, then $n=7$ and $\Gamma=G_7$; and in Figure~\ref{fig:lemma3.9.6c}, if $z\sim w$, $5\sim x$, and $5\sim y$, then $n=9$. By $\sw(5, x,3, y)$ and then $\sw(3, 5,4, z)$, we can thus transfer $\Gamma$ to $G_9$.

\begin{figure}[h!]
\centering
\subfloat[]{\begin{tikzpicture}[scale=.9]
	\vertex[fill] (1) at (.6,-.4) [label=left:\scriptsize{$1$}] {};
	\vertex[fill] (2) at (1.5,.2) [label=above:\scriptsize{$2$}] {};
	\vertex[fill] (3) at (1.5,-1) [label=below:\scriptsize{$3  $}] {};
	\vertex[fill] (4) at (3.3,.1) [label=above:\scriptsize{$4$}] {};
	\vertex[fill] (5) at (3.3,-.9) [label=below:\scriptsize{$5$}] {};
	\vertex[fill] (x) at (2.3,-1.2) [label=above:\scriptsize{$x$}] {};
	\vertex[fill] (y) at (1.9,-1.6) [label=below:\scriptsize{$y$}] {};
\vertex[fill] (z) at (2.8,-1.6) [label=below:\scriptsize{$z$}] {};
	\path
		(1) edge (2)
		(1) edge (3)
		(1) edge (4)
		(1) edge (5)
	    (2) edge (3)
	    (2) edge (5)
	    (2) edge (4)
		(3) edge (y)
		(x) edge (3)
	   (4) edge (z)
	   (4) edge (x)
	   (x) edge (z)
	   (x) edge (y)
	   (y) edge (z) ;
\end{tikzpicture}\label{fig:lemma3.9.6a}}
\quad \quad \quad
\subfloat[]{\begin{tikzpicture}[scale=.9]
	\vertex[fill] (1) at (.6,-.4) [label=left:\scriptsize{$1$}] {};
	\vertex[fill] (2) at (1.5,.2) [label=above:\scriptsize{$2$}] {};
	\vertex[fill] (3) at (1.5,-1) [label=below:\scriptsize{$3  $}] {};
	\vertex[fill] (4) at (3.3,.1) [label=above:\scriptsize{$4$}] {};
	\vertex[fill] (5) at (3.3,-.9) [label=below:\scriptsize{$5$}] {};
	\vertex[fill] (x) at (2.3,-1.2) [label=above:\scriptsize{$x$}] {};
	\vertex[fill] (y) at (2.3,-1.7) [label=below:\scriptsize{$y$}] {};
	\path
		(1) edge (2)
		(1) edge (3)
		(1) edge (4)
		(1) edge (5)
	    (2) edge (3)
	    (2) edge (5)
	    (2) edge (4)
		(3) edge (x)
		(4) edge (y)
	    (4) edge (x)
	    (x) edge (y)
	    (3) edge (y)   ;
\end{tikzpicture}\label{fig:lemma3.9.6b}}
\quad \quad \quad
\subfloat[]{\begin{tikzpicture}[scale=.9]
	\vertex[fill] (1) at (.6,-.4) [label=left:\scriptsize{$1$}] {};
	\vertex[fill] (2) at (1.5,.2) [label=above:\scriptsize{$2$}] {};
	\vertex[fill] (3) at (1.5,-1) [label=below:\scriptsize{$3  $}] {};
	\vertex[fill] (4) at (3.3,.1) [label=above:\scriptsize{$4$}] {};
	\vertex[fill] (5) at (3.3,-.9) [label=below:\scriptsize{$5$}] {};
	\vertex[fill] (x) at (2.1,-1.2) [label=above:\scriptsize{$x$}] {};
	\vertex[fill] (y) at (1.9,-1.6) [label=below:\scriptsize{$y$}] {};
    \vertex[fill] (z) at (2.6,-1.2) [label=above:\scriptsize{$z$}] {};
	\vertex[fill] (w) at (2.8,-1.6) [label=below:\scriptsize{$w$}] {};
	\path
		(1) edge (2)
		(1) edge (3)
		(1) edge (4)
		(1) edge (5)
	    (2) edge (3)
	    (2) edge (4)
	    (2) edge (5)
		(3) edge (x)
		(3) edge (y)
		(4) edge (z)
	   (4) edge (w)
	   (x) edge (w)
	   (x) edge (z)
       (y) edge (w)
       (y) edge (z);
\end{tikzpicture}\label{fig:lemma3.9.6c}}
\caption{Some possible situations in Step 4}\label{fig:lemma3.9.6}
\end{figure}
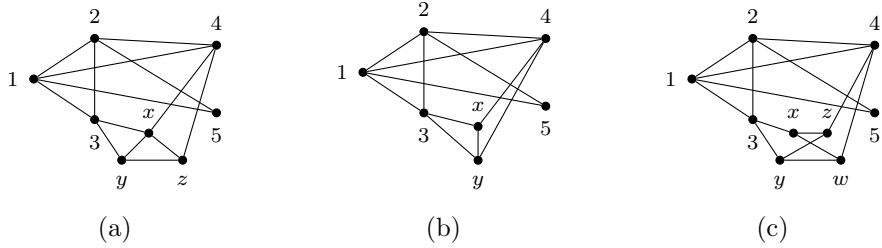

\noindent\textbf{Step 5.} So far we obtain on the first five vertices either the subgraph (a) or (b) of Figure~\ref{fig:lemma3.9.7}.
If (a) is the case, letting $x$ to be the fourth neighbor of $3$, then $\sw(3,x,5,4)$  connects $3$ to $5$, and so we obtain $D'_4$.
  Now, assume that (b) is the case. If we can find a switch to connect $3$ to $5$,  we again reach  $D'_4$.
Otherwise, it is easily seen that by proper switching we can  connect $3$ to $6$ as shown in Figure~\ref{fig:lemma3.9.7c}.
Furthermore, if we can find a suitable switch to connect $4$ to $6$, we reach the graph $H_1$ of Figure~\ref{fig:H1H2}.
Otherwise, it is easily seen  by switching that $4\sim 7$  and that we can reach Figure~\ref{fig:lemma3.9.7d}.
 If there is no switch to connect $4$ to $6$, then we can find a proper switch to   connect $5$ to $6$, except when all the three vertices $5$, $6$, and $7$ are adjacent to $8$ and $9$ and $8\sim 9$, in which case
  $n=9$ and $\Gamma=G_9$. Now, from (d), if $5\nsim 7$, then $\sw(4,7,6,5)$  connects $4$ to $6$ and thus $H_1$ is obtained. Otherwise
  $5\sim 7$   and we reach the graph $H_2$ of Figure~\ref{fig:H1H2}.

\begin{figure}[h!]
\centering
\subfloat[]
{\begin{tikzpicture}
	\vertex[fill] (r) at (0,0) [label=left:\scriptsize{$1$}] {};
	\vertex[fill] (r1) at (.5,.5) [label=above:\scriptsize{$2$}] {};
	\vertex[fill] (r2) at (.5,-.5) [label=below:\scriptsize{$3$}] {};
	\vertex[fill] (r3) at (1.5,.5) [label=above:\scriptsize{$4$}] {};
	\vertex[fill] (r4) at (1.5,-.5) [label=below:\scriptsize{$5$}] {};
   	\path
		(r) edge (r1)
		(r) edge (r2)
	   (r) edge (r3)
	   (r) edge (r4)
		(r1) edge (r2)
		(r1) edge (r3)
        (r1) edge (r4)
		(r2) edge (r3)
	    (r3) edge (r4);		
\end{tikzpicture}\label{fig:lemma3.9.7a}}
\quad \quad \quad
\subfloat[]{\begin{tikzpicture}
	\vertex[fill] (r) at (0,0) [label=left:\scriptsize{$1$}] {};
	\vertex[fill] (r1) at (.5,.5) [label=above:\scriptsize{$2$}] {};
	\vertex[fill] (r2) at (.5,-.5) [label=below:\scriptsize{$3$}] {};
	\vertex[fill] (r3) at (1.5,.5) [label=above:\scriptsize{$4$}] {};
	\vertex[fill] (r4) at (1.5,-.5) [label=below:\scriptsize{$5$}] {};
	\path
		(r) edge (r1)
		(r) edge (r2)
	   (r) edge (r3)
	   (r) edge (r4)
		(r1) edge (r2)
		(r1) edge (r3)
        (r1) edge (r4)
		(r2) edge (r3);
\end{tikzpicture}\label{fig:lemma3.9.7b}}
\quad \quad \quad
\subfloat[]{\begin{tikzpicture}
	\vertex[fill] (r1) at (0,0) [label=below:\scriptsize{$1$}] {};
	\vertex[fill] (r2) at (0,.8) [label=above:\scriptsize{$2$}] {};
	\vertex[fill] (r3) at (.8,1) [label=above:\scriptsize{$3$}] {};
	\vertex[fill] (r4) at (.8,-.2) [label=below:\scriptsize{$4$}] {};
    \vertex[fill] (r5) at (1.3,.4) [label=right:\scriptsize{$5$}] {};
    \vertex[fill] (r7) at (2,.8) [label=above:\scriptsize{$6$}] {};
	\path
		(r5) edge (r2)
		(r1) edge (r2)
	    (r1) edge (r5)
		(r1) edge (r3)
        (r1) edge (r4)
		(r2) edge (r3)
	    (r2) edge (r4)
	    (r4) edge (r3)
	    (r3) edge (r7) ;
\end{tikzpicture}\label{fig:lemma3.9.7c}}
\quad \quad \quad
\subfloat[]{\begin{tikzpicture}
	\vertex[fill] (r1) at (0,0) [label=below:\scriptsize{$1$}] {};
	\vertex[fill] (r2) at (0,.8) [label=above:\scriptsize{$2$}] {};
	\vertex[fill] (r3) at (.8,1) [label=above:\scriptsize{$3$}] {};
	\vertex[fill] (r4) at (.8,-.2) [label=below:\scriptsize{$4$}] {};
    \vertex[fill] (r5) at (1.3,.4) [label=right:\scriptsize{$5$}] {};
    \vertex[fill] (r7) at (2,.8) [label=above:\scriptsize{$6$}] {};
    \vertex[fill] (r6) at (2,0) [label=below:\scriptsize{$7$}] {};
	\path
		(r5) edge (r2)
		(r1) edge (r2)
	    (r1) edge (r5)
		(r1) edge (r3)
        (r1) edge (r4)
		(r2) edge (r3)
	    (r2) edge (r4)
	    (r4) edge (r3)
	    (r3) edge (r7)
	     (r4) edge (r6) ;
\end{tikzpicture}\label{fig:lemma3.9.7d}}
\caption{Some possible situations in Step 5}\label{fig:lemma3.9.7}
\end{figure}
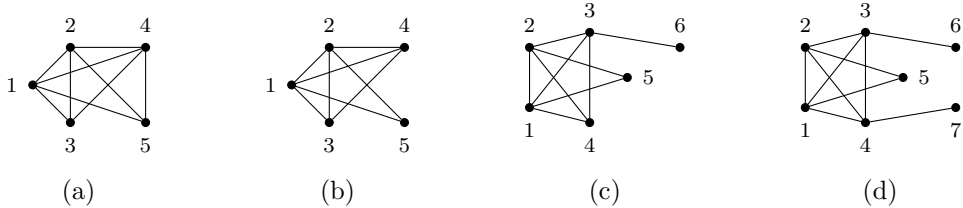

\noindent\textbf{Step 6.} So far we have obtained  one of the subgraphs $D'_4$, or $H_1,H_2$ of Figure~\ref{fig:H1H2}, unless $n\le9$, in which case we obtained the graphs $G_i$ of Figure~\ref{fig:G5-G9}. We show that
 continuous reconnecting, starting from $H_1$ and $H_2$, leads to $D_1$, $D_2$, $D'_3$, or $D'_4$.

First, consider $H_1$. We have either $5\sim 6$ or $5\nsim 6$. Let $5\sim 6$. It is easy to find a switch that  connects $5$ to $7$.
If further $6\sim 7$, then we have the block $D_1$. If $6\nsim 7$,
it is easily seen, by switching, that $6\sim 8$.  Then  $\sw(3, 6, 5, 7)$ reduces the subgraph on $\{1,\ldots,5\}$ to $D'_4$.
Now, let $5\nsim 6$. By switching it is seen that $5\sim 7$ and $5\sim 8$.
 Thus we are in the situation of Lemma~\ref{lem:r+1,r+2VerCut}  for $r=4$, which leads to  either of the  subgraphs (a) or (b) of Figure~\ref{fig:lemma3.9.8}.
Now, $\sw(3, 6, 5, x)$ reduces (b) to the subgraph of Figure~\ref{fig:lemma3.9.8c}, which is the unique graph of Theorem~\ref{thm:quartic} on 10 vertices.
For (a), first let $6\sim 8$. If further $7\sim 8$, then we get $D'_3$, otherwise $\sw(5, 8, 6, 7)$ and then $\sw(3, 6, 5, 7)$ reduce the subgraph on $\{1,\ldots,5\}$ to $D'_4$. Now, let  $6\nsim 8$.  Then $\sw(3, 6, 5, 8)$ reduces the subgraph on $\{1,\ldots,5\}$ to $D'_4$.

\begin{figure}[h!]
\centering
\subfloat[]{\begin{tikzpicture} [scale=.9]
	\vertex[fill] (r1) at (0,1) [label=above:\scriptsize{$1$}] {};
	\vertex[fill] (r2) at (.8,1.2) [label=above:\scriptsize{$2$}] {};
	\vertex[fill] (r3) at (0,0) [label=below:\scriptsize{$3$}] {};
	\vertex[fill] (r4) at (.8,-.2) [label=below:\scriptsize{$4$}] {};
    \vertex[fill] (r5) at (1.6,1) [label=above:\scriptsize{$5$}] {};
    \vertex[fill] (r6) at (1.6,0) [label=below:\scriptsize{$6$}] {};
    \vertex[fill] (r7) at (2.6,1) [label=above:\scriptsize{$7$}] {};
    \vertex[fill] (r8) at (2.6,0) [label=below:\scriptsize{$8$}] {};
	\path
		(r5) edge (r2)
		(r1) edge (r2)
	    (r1) edge (r5)
		(r1) edge (r3)
        (r1) edge (r4)
		(r2) edge (r3)
	    (r2) edge (r4)
	    (r3) edge (r6)
	    (r4) edge (r3)
	    (r4) edge (r6)
	    (r5) edge (r8)
	    (r5) edge (r7)
	    (r6) edge (r7) ;
	  \end{tikzpicture}\label{fig:lemma3.9.8a}}
\quad   \quad    \quad
\subfloat[]{\begin{tikzpicture}[scale=.9]
	\vertex[fill] (r1) at (0,1) [label=above:\scriptsize{$1$}] {};
	\vertex[fill] (r2) at (.8,1.2) [label=above:\scriptsize{$2$}] {};
	\vertex[fill] (r3) at (0,0) [label=below:\scriptsize{$3$}] {};
	\vertex[fill] (r4) at (.8,-.2) [label=below:\scriptsize{$4$}] {};
    \vertex[fill] (r5) at (1.6,1) [label=above:\scriptsize{$5$}] {};
    \vertex[fill] (r6) at (1.6,0) [label=below:\scriptsize{$6$}] {};
    \vertex[fill] (r7) at (2.4,1.2) [] {};
    \vertex[fill] (r8) at (2.4,-.2) [] {};
    \vertex[fill] (r9) at (3.2,1) [label=above:\scriptsize{$x$}] {};
    \vertex[fill] (r10) at (3.2,0) [] {};
	\path
		(r5) edge (r2)
		(r1) edge (r2)
	    (r1) edge (r5)
		(r1) edge (r3)
        (r1) edge (r4)
		(r2) edge (r3)
	    (r2) edge (r4)
	    (r3) edge (r6)
	    (r4) edge (r3)
	    (r4) edge (r6)
	    (r5) edge (r9)
	    (r5) edge (r7)
	    (r6) edge (r8)
	    (r6) edge (r10)
	    (r7) edge (r8)
	    (r7) edge (r9)
	    (r7) edge (r10)
	    (r8) edge (r9)
	    (r8) edge (r10)
	    (r9) edge (r10);
\end{tikzpicture}\label{fig:lemma3.9.8b}}
\quad   \quad   \quad
\subfloat[]{\begin{tikzpicture}[scale=.9]
	\vertex[fill] (r) at (0,0) [] {};
	\vertex[fill] (r1) at (.5,.5) [] {};
	\vertex[fill] (r2) at (.5,-.5) [] {};
	\vertex[fill] (r3) at (1.5,.5) [] {};
	\vertex[fill] (r4) at (1.5,-.5) [] {};
    \vertex[fill] (r5) at (2.5,.5) [] {};
	\vertex[fill] (r6) at (2.5,-.5) [] {};
    \vertex[fill] (r7) at (3.5,.5) [] {};
	\vertex[fill] (r8) at (3.5,-.5) [] {};
    \vertex[fill] (r9) at (4,0) []{};
    \tikzstyle{vertex}=[circle, draw, inner sep=0pt, minimum size=0pt]
    \vertex[fill] () at (3.5,-1.2) [] {};
	\path
		(r) edge (r1)
		(r) edge (r2)
	    (r) edge (r3)
	    (r) edge (r4)
		(r1) edge (r2)
		(r1) edge (r3)
        (r1) edge (r4)
		(r2) edge (r3)
	    (r2) edge (r4)
	    (r4) edge (r6)
	    (r5) edge (r3)
	    (r5) edge (r7)
	    (r5) edge (r8)
	    (r6) edge (r7)
	     (r6) edge (r8)
	     (r7) edge (r8)
	     (r9) edge (r8)
	     (r9) edge (r7)
	     (r9) edge (r6)
	     (r9) edge (r5);
\end{tikzpicture}\label{fig:lemma3.9.8c}}
\caption{Some possible situations in Step 6 }\label{fig:lemma3.9.8}
\end{figure}
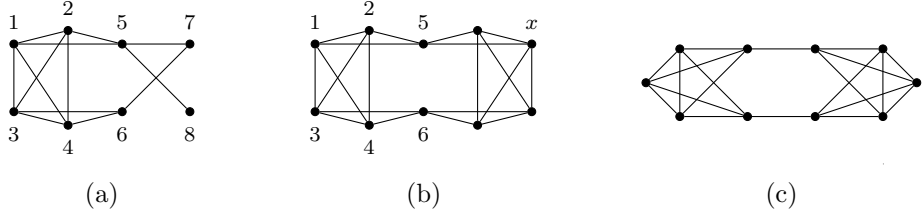

Secondly, consider $H_2$.  We have either $6\sim 7$ or $6\nsim 7$. First let $6\sim 7$. It is easy to find a switch that  connects $6$ to $8$.
If further $7\sim 8$, then we obtain the block $D_2$. If $7\nsim 8$, then  $\sw(4, 7, 6, 8)$ reduces the graph to $D_1$.
Now, let $6\nsim 7$.  Then  $\sw(3, 6, 5, 7)$ reduces the subgraph on $\{1,\ldots,5\}$ to $D'_4$.

\end{proof}

\subsubsection{General Steps}

In this section, we continue reconnecting  $\Gamma$  by proper switchings to construct the middle and end blocks with the structure described in Theorem~\ref{thm:quartic}.

\begin{lemma}\label{lem:begin2}
	Let $H$ be a graph with vertices $r,\ldots,r+\ell-1$ where all the vertices have degree $4$ except for $r$, which has degree $2$.
If $\ell\ge10$,  then, by proper switchings, we can  transform the induced subgraph on the first four or five vertices  into one of the subgraphs given in Figure~$\ref{fig:cut}$. If $\ell\le9$, then $H$ can be transformed into either of   $\tilde D_1,\tilde D_2,\tilde D_3$ or $\tilde D_4$.
\end{lemma}
\begin{figure}[h!]
\centering
\subfloat[]{\begin{tikzpicture}[scale=.9]
   	\vertex[fill] (r) at (0,0) [label=left:\scriptsize{$r$}] {};
	\vertex[fill] (r1) at (.5,.6) [label=above:\scriptsize{$r+1\  \  $}] {};
	\vertex[fill] (r2) at (.5,-.6) [label=below:\scriptsize{$r+2\  \ $}] {};
	\vertex[fill] (r3) at (1.7,.6) [label=above:\scriptsize{$ \  \   r+3$}] {};
	\vertex[fill] (r4) at (1.7,-.6) [label=below:\scriptsize{$  \  \    r+4$}] {};
	\path
		(r) edge (r1)
		(r) edge (r2)
		(r1) edge (r2)
		(r1) edge (r3)
        (r1) edge (r4)
		(r2) edge (r3)
	    (r2) edge (r4);
	\end{tikzpicture}\label{fig:cuta}}
\quad \quad
\subfloat[]{\begin{tikzpicture}[scale=.9]
	\vertex[fill] (r) at (0,0) [label=left:\scriptsize{$r$}] {};
	\vertex[fill] (r1) at (.5,.6) [label=above:\scriptsize{$r+1 \  \ $}] {};
	\vertex[fill] (r2) at (.5,-.6) [label=below:\scriptsize{$r+2 \  \ $}] {};
	\vertex[fill] (r3) at (1.7,.6) [label=above:\scriptsize{$ \  \  r+3$}] {};
	\vertex[fill] (r4) at (1.7,-.6) [label=below:\scriptsize{$ \  \  r+4$}] {};
	\path
		(r) edge (r1)
		(r) edge (r2)
		(r1) edge (r2)
		(r1) edge (r3)
        (r1) edge (r4)
		(r2) edge (r3)
	    (r2) edge (r4)
	    (r3) edge (r4);
	\end{tikzpicture}\label{fig:cutb}}
\quad \quad
\subfloat[]{\begin{tikzpicture}[scale=.9]
	\vertex[fill] (r) at (1,0) [label=left:\scriptsize{$r$}] {};
	\vertex[fill] (r1) at (1.5,.6) [label=above:\scriptsize{$r+1 \  \  $}] {};
	\vertex[fill] (r2) at (1.5,-.6) [label=below:\scriptsize{$r+2 \  \ $}] {};
	\vertex[fill] (r3) at (2.48,0) [label=above:\scriptsize{$r+3$}] {};
	\vertex[fill] (r4) at (3.4,.6) [label=above:\scriptsize{$  \   \   r+4$}] {};
    \vertex[fill] (r5) at (3.4,-.6) [label=below:\scriptsize{$  \   \  r+5$}] {};
\path
		(r) edge (r1)
		(r) edge (r2)
		(r1) edge (r2)
		(r1) edge (r3)
        (r1) edge (r4)
		(r2) edge (r3)
	    (r2) edge (r5)
	    (r3) edge (r4)
	    (r3) edge (r5)
	    (r5) edge (r4);
\end{tikzpicture}\label{fig:cutc}}\quad \quad
\subfloat[]{\begin{tikzpicture}[scale=.9]
  	\vertex[fill] (r) at (1,0) [label=left:\scriptsize{$r$}] {};
	\vertex[fill] (r1) at (1.5,.6) [label=above:\scriptsize{$r+1 \  \  \  $}] {};
	\vertex[fill] (r2) at (1.5,-.6) [label=below:\scriptsize{$r+2 \  \  \  $}] {};
	\vertex[fill] (r3) at (2.1,0) [label=right:\scriptsize{$r+3$}] {};
	\vertex[fill] (r4) at (2.7,.6) [label=above:\scriptsize{$\  \  \  r+4$}] {};
    \vertex[fill] (r5) at (2.7,-.6) [label=below:\scriptsize{$\  \  \  r+5$}] {};
	\path
		(r) edge (r1)
		(r) edge (r2)
		(r1) edge (r2)
		(r1) edge (r3)
        (r1) edge (r4)
		(r2) edge (r3)
	    (r2) edge (r5)
	    (r3) edge (r4)
	    (r3) edge (r5);
	  \end{tikzpicture}\label{fig:cutd}}
\caption{Subgraphs which can be constructed on first few vertices of $H$ in Lemma~\ref{lem:begin2}}
\label{fig:cut}
\end{figure}
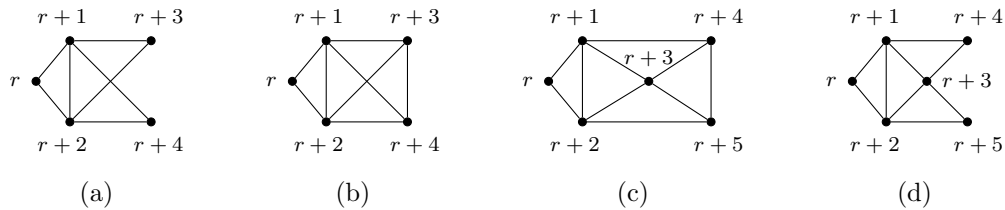

\begin{proof}
First we show that $r\sim r+1,r+2$. To see this, assume that $r$ is not adjacent with some $x\in\{r+1,r+2\}$. So $r$ has some neighbor $y>r+2$.
It is not possible that $y$ is adjacent with any neighbor of $x$ (this requires $y$ having degree $5$).
So there is some vertex $z$ such that $z\sim x$ and $z\nsim y$. Now, $\sw(r, y, x, z)$ makes $r$ adjacent to $x$.
The graph  $H-r$ satisfies the conditions of Lemma \ref{lem:r+1--r+2} with $m=\ell-1$, and so  if $\ell=6,9$,  $H-r$ can be transferred into  $\tilde D'_4, \tilde D'_3$  which means that $H$  can be transferred into $\tilde D_4, \tilde D_3$, respectively.
For $\ell=6,9$, nothing remains to be proved and so we assume that $\ell\ne6,9$.
 Hence, from Lemma \ref{lem:r+1--r+2}, it follows that $r+1\sim r+2$.
By the same argument given above for $r$, we see that  $r + 1$ is adjacent with both $r + 3$ and $r+4$.
 Now $H-r$ satisfies the conditions of Lemma \ref{lem:r+1,r+2VerCut}  and so if $\ell=7$ (i.e. $m=6$), $H-r$ can be transferred into $\bar D_1$ (with
  $r+1\sim r+2$) and so $H$ to  $\tilde D_1$. Therefore, we assume that $\ell\ne7$.
 Thus from Lemma \ref{lem:r+1,r+2VerCut}  it follows that $r+2\sim r+3$.
 So far we have obtained the desired subgraph on $r,\ldots,r+3$, which is
 the same on all the graphs of Figure \ref{fig:cut}.

To conclude the proof, we consider the  possible adjacencies between the three vertices $r+2,r+3$, and $r+4$.
If $r+2\sim r+4$, we are done as we obtain either the subgraph (a) or (b) of Figure~\ref{fig:cut}.
So, in what follows we assume that  $r+2\nsim r+4$. We claim that $r+3\sim r+4$.
By Lemma~\ref{lem:connected}, we may assume that $H\setminus\{r,r+1,r+2\}$ is connected.
Let $x\neq r, r+1, r+3$ be the fourth neighbor of $r+2$.
If $x\sim r+4$, the  desired switch is available, except in the case (a) of Figure~\ref{fig:lemma3.10.6} (in which case $m=8$ and the block $\tilde D_2$ is obtained). If $x\nsim r+4$,	by Remark~\ref{remark}  the desired switch is available in any situation other than the case (b) of Figure~\ref{fig:lemma3.10.6}.
	Then either $y\nsim z$ or $y\nsim w$ for which by either $\sw(r+3, y, r+4, z)$ or $\sw(r+3, y, r+4, w)$, respectively, we have $r+3\sim r+4$ and the claim follows.
  Again we may assume that $H\setminus\{r,\ldots,r+3\}$ is connected.
Our goal is to show that $r+2\sim r+5$ and $r+3\sim r+5$ and thus we will come up with either of the subgraphs (c) or (d) of Figure \ref{fig:cut}. As above $x\neq r, r+1, r+3$ is the fourth neighbor of $r+2$.
 There are two possibilities:
\begin{itemize}
\item[(i)]  $x\sim r+4$. A switch to connect $r+2$ to $r+4$ exists, except in the situation of Figure~\ref{fig:lemma3.10.6a}.
If $x=r+5$, then we are done by reaching  the subgraph given in Figure~\ref{fig:cutc}.
If $y=r+5$, let  $z$ and $w$ be the other neighbors of $r+5$. Then $\sw(r+2,x,r+5,z)$ and $\sw(r+3,x,r+5,w)$ give rise to the subgraph of Figure~\ref{fig:cutc}  again.
If $y\neq r+5$ and $x\neq r+5$,  then $r+5$ has two neighbors $z$ and $w$ that are non-adjacent to $x$. Then by $\sw(r+2,x,r+5,z)$ and $\sw(r+3,x,r+5,w)$ the subgraph  of Figure~\ref{fig:cutd} is obtained.
\item[(ii)]   $x\nsim r+4$. A switch to connect $r+2$ to $r+4$ exists, except in the situation of Figure~\ref{fig:lemma3.10.6b}.
If $x=r+5$, then we are done by reaching  the subgraph of Figure~\ref{fig:cutd}.
Otherwise, in a similar manner, the switches which give rise to the subgraph of Figure~\ref{fig:cutd} can be found easily by examining the  adjacencies between the neighbors of $r+5$ and $r+2$ (or $r+3$).
\end{itemize}
\end{proof}
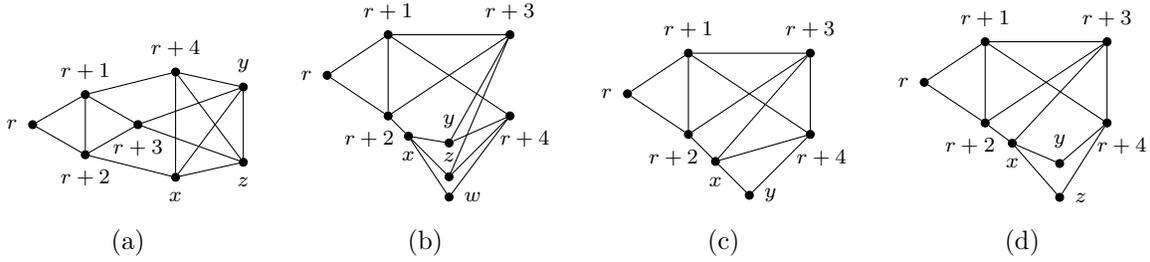
\begin{figure}[h!]
\centering
\subfloat[]{\begin{tikzpicture}
    \vertex[fill] (r5) at (2.5,0) [label=left:\scriptsize{$r$}] {};
	\vertex[fill] (r6) at (3.2,.4) [label=above:\scriptsize{$r+1$}] {};
    \vertex[fill] (r7) at (3.2,-.4) [label=below:\scriptsize{$r+2$}] {};
	\vertex[fill] (r8) at (3.9,0) [label=below:\scriptsize{$r+3$}] {};
    \vertex[fill] (r9) at (4.4,.7) [label=above:\scriptsize{$r+4$}] {};
	\vertex[fill] (r10) at (4.4,-.7) [label=below:\scriptsize{$x$}] {};
    \vertex[fill] (r11) at (5.3,.5) [label=above:\scriptsize{$y$}] {};
	\vertex[fill] (r12) at (5.3,-.5) [label=below:\scriptsize{$z$}] {};
	\path
        (r5) edge (r6)
	    (r5) edge (r7)	
	    (r6) edge (r7)	
	    (r10) edge (r7)
	    (r9) edge (r6)	
	     (r8) edge (r7)
	    (r8) edge (r6)	
        (r12) edge (r8)
	    (r11) edge (r8)	
	    (r9) edge (r10)
	    (r9) edge (r11)
	    (r9) edge (r12)
	    (r10) edge (r11)
	    (r10) edge (r12)
	    (r11) edge (r12);
\end{tikzpicture}\label{fig:lemma3.10.6c}}
	\quad
\subfloat[]{\begin{tikzpicture}[scale=.9]
		\vertex[fill] (r) at (.6,-.4) [label=left:\scriptsize{$r$}] {};
		\vertex[fill] (r1) at (1.5,.2) [label=above:\scriptsize{$r+1$}] {};
		\vertex[fill] (r2) at (1.5,-1) [label=below:\scriptsize{$r+2 \  \ \ \ \  $}] {};
		\vertex[fill] (r3) at (3.3,.2) [label=above:\scriptsize{$r+3$}] {};
		\vertex[fill] (r4) at (3.3,-1) [label=below:\scriptsize{$\ \ \ \ r+4$}] {};
		\vertex[fill] (x) at (1.8,-1.3) [label=below:\scriptsize{$x$}] {};
		\vertex[fill] (y) at (2.4,-1.4) [label=above:\scriptsize{$y$}] {};
		\vertex[fill] (z) at (2.4,-1.9) [label=above:\scriptsize{$z$}] {};
		\vertex[fill] (w) at (2.4,-2.2) [label=right:\scriptsize{$w$}] {};
				\path
		(r) edge (r1)
		(r) edge (r2)
		(r1) edge (r2)
		(r1) edge (r3)
		(r1) edge (r4)
		(r2) edge (r3)
		(r2) edge (x)
		(x) edge (y)
		(x) edge (z)
		(x) edge (w)
		(y) edge (r3)
		(y) edge (r4)
		(r3) edge (z)
		(r4) edge (z)
		(w) edge (r4);
				\end{tikzpicture}\label{fig:lemma3.10.6d}}
	\quad
	\subfloat[]{\begin{tikzpicture}[scale=.9]
	\vertex[fill] (r) at (.6,-.4) [label=left:\scriptsize{$r$}] {};
	\vertex[fill] (r1) at (1.5,.2) [label=above:\scriptsize{$r+1$}] {};
	\vertex[fill] (r2) at (1.5,-1) [label=below:\scriptsize{$r+2 \ \ \ \  $}] {};
	\vertex[fill] (r3) at (3.3,.2) [label=above:\scriptsize{$r+3$}] {};
	\vertex[fill] (r4) at (3.3,-1) [label=below:\scriptsize{$\ \ \ r+4$}] {};
	\vertex[fill] (x) at (1.9,-1.4) [label=below:\scriptsize{$x$}] {};
	\vertex[fill] (y) at (2.4,-1.9) [label=right:\scriptsize{$y$}] {};
   	\path
		(r) edge (r1)
		(r) edge (r2)
		(r1) edge (r2)
		(r1) edge (r3)
	    (r1) edge (r4)
		(r2) edge (r3)
	    (r4) edge (r3)
		(x) edge (r2)
		(x) edge (r3)
		(x) edge (r4)
	    (y) edge (x)
	    (y) edge (r4);
\end{tikzpicture}\label{fig:lemma3.10.6a}}
\quad
\subfloat[]{\begin{tikzpicture}[scale=.9]
	\vertex[fill] (r) at (.6,-.4) [label=left:\scriptsize{$r$}] {};
	\vertex[fill] (r1) at (1.5,.2) [label=above:\scriptsize{$r+1$}] {};
	\vertex[fill] (r2) at (1.5,-1) [label=below:\scriptsize{$r+2 \ \ \ \  $}] {};
	\vertex[fill] (r3) at (3.3,.2) [label=above:\scriptsize{$r+3$}] {};
	\vertex[fill] (r4) at (3.3,-1) [label=below:\scriptsize{$\ \ \ \   r+4$}] {};
	\vertex[fill] (x) at (1.9,-1.3) [label=below:\scriptsize{$x$}] {};
	\vertex[fill] (y) at (2.6,-1.6) [label=above:\scriptsize{$y$}] {};
	\vertex[fill] (z) at (2.6,-2.1) [label=right:\scriptsize{$z$}] {};
	\path
		(r) edge (r1)
		(r) edge (r2)
		(r1) edge (r2)
		(r1) edge (r3)
	    (r1) edge (r4)
	    (r2) edge (x)
	    (r2) edge (r3)
		(r3) edge (x)
	    (r3) edge (r4)
	    (r4) edge (y)
	    (r4) edge (z)
		(x) edge (y)
		(x) edge (z);
	 \end{tikzpicture}\label{fig:lemma3.10.6b}}
\caption{Some possible situations in the proof of Lemma~\ref{lem:begin2}}\label{fig:lemma3.10.6}
\end{figure}

We are now in a position to prove the `first half'  of Theorem~\ref{thm:quartic}, that is to show that $\Gamma$ can be transferred to one of the graphs of Theorem~\ref{thm:quartic}.

\begin{theorem}\label{thm:quarticTOmaximal}
By proper switchings, any connected quartic graph can be turned into one of the graphs described in Theorem~\ref{thm:quartic}.
\end{theorem}
\begin{proof}
For $n\le9$ the assertion is proved in Lemma~\ref{lem:begin1}. So we may assume that $n\ge10$.
We start rebuilding $\Gamma$ on its first few vertices  as in Lemma~\ref{lem:begin1}. As we saw there, the first few vertices of $\Gamma$ can be  transformed into one of the subgraphs $D_1,D_2,D'_3,D'_4$. Moreover, whatever we  obtained, we ended up either with a cut vertex, or with one of the situations (i) or (ii) of Table~\ref{tab:v,v+1}.
Since cut vertices in a quartic graph have necessarily degree $2$, we can employ Lemma~\ref{lem:begin2} for
 reconnecting following a cut vertex. Doing so,   we again reach at one of the  situations (i), (ii), or (iii) of Table~\ref{tab:v,v+1}.
We now demonstrate what can be constructed  afterwards. As verified below, by proper switchings,
 the situation of the next few vertices can be determined from the situation of $v,v+1$ according to Table~\ref{tab:v,v+1}.
\begin{table}[h!]
 	\centering
 	\begin{tabular}{cc}
 		\hline
 		 situation of $v,v+1$  &   situation of next few vertices after appropriate switchings \\ \hline
  (i)~~$\begin{array}{c}\begin{tikzpicture}[scale=.7]
	\vertex[fill] (r) at (0,0) [label=above:\scriptsize{$v$}] {};
	\vertex[fill] (r1) at (0,-1) [label=below:\scriptsize{$v+1$}] {};
    \tikzstyle{vertex}=[circle, draw, inner sep=0pt, minimum size=0pt]
    \vertex (s) at (-.4,0) [] {};
	\vertex (ss) at (-.4,-.3) [] {};
    \vertex (sss) at (-.4,-.7) [] {};
	\vertex (ssss) at (-.4,-1) [] {};
	\path
        (r) edge (r1)
		(r) edge (s)
		(r) edge (ss)
		(r1) edge (sss)
		(r1) edge (ssss);
\end{tikzpicture} \end{array}$
&
$\begin{array}{cccc}
\begin{tikzpicture}[scale=.7]
	\vertex[fill] (r) at (0,0) [label=above:\scriptsize{$v$}] {};
	\vertex[fill] (r1) at (0,-1) [label=below:\scriptsize{$v+1$}] {};
    \vertex[fill] (r2) at (.5,-.5) [label=right:\scriptsize{$v+2$}] {};
    \tikzstyle{vertex}=[circle, draw, inner sep=0pt, minimum size=0pt]
    \vertex (s) at (-.4,0) [label=right:$$] {};
	\vertex (ss) at (-.4,-.3) [label=right:$$] {};
    \vertex (sss) at (-.4,-.7) [label=right:$$] {};
	\vertex (ssss) at (-.4,-1) [label=right:$$] {};
	\path
        (r) edge (r1)
         (r1) edge (r2)
          (r) edge (r2)
		(r) edge (s)
		(r) edge (ss)
		(r1) edge (sss)
		(r1) edge (ssss);
\end{tikzpicture}
&
    \begin{tikzpicture}[scale=.7]
	\vertex[fill] (r3) at (1.5,.5) [label=above:\scriptsize{$v$}] {};
	\vertex[fill] (r4) at (1.5,-.5) [label=below:\scriptsize{$v+1$}] {};
    \vertex[fill] (r5) at (2.8,.5) [label=above:\scriptsize{$v+2$}] {};
	\vertex[fill] (r6) at (2.8,-.5) [label=below:\scriptsize{$v+3$}] {};
\tikzstyle{vertex}=[circle, draw, inner sep=0pt, minimum size=0pt]
    \vertex (s) at (1.1,.5) [label=right:$$] {};
	\vertex (ss) at (1.1,.2) [label=right:$$] {};
    \vertex (sss) at (1.1,-.2) [label=right:$$] {};
	\vertex (ssss) at (1.1,-.5) [label=right:$$] {};
	\path
	    (r3) edge (r4)
	    (r5) edge (r3)
	    (r4) edge (r6)
	    (r5) edge (r6)
	    (r3) edge (s)
		(r3) edge (ss)
		(r4) edge (sss)
		(r4) edge (ssss);
\end{tikzpicture}
&
 \begin{tikzpicture}[scale=.7]
	\vertex[fill] (r3) at (1.5,.5) [label=above:\scriptsize{$v$}] {};
	\vertex[fill] (r4) at (1.5,-.5) [label=below:\scriptsize{$v+1$}] {};
    \vertex[fill] (r5) at (2.5,.5) [] {};
	\vertex[fill] (r6) at (2.5,-.5) [] {};
    \vertex[fill] (r7) at (3.5,.5) [] {};
	\vertex[fill] (r8) at (3.5,-.5) [] {};
    \vertex[fill] (r9) at (4,0) [] {};
    \tikzstyle{vertex}=[circle, draw, inner sep=0pt, minimum size=0pt]
    \vertex (s) at (1.1,.5) [label=right:$$] {};
	\vertex (ss) at (1.1,.2) [label=right:$$] {};
    \vertex (sss) at (1.1,-.2) [label=right:$$] {};
	\vertex (ssss) at (1.1,-.5) [label=right:$$] {};
	\path
	    (r4) edge (r6)
	    (r5) edge (r3)
	    (r5) edge (r7)
	    (r5) edge (r8)
	    (r6) edge (r7)
	     (r6) edge (r8)
	     (r7) edge (r8)
	     (r9) edge (r8)
	     (r9) edge (r7)
	     (r9) edge (r6)
	     (r9) edge (r5)
	      (r3) edge (r4)
	      (r3) edge (s)
		(r3) edge (ss)
		(r4) edge (sss)
		(r4) edge (ssss);
\end{tikzpicture}
&
\begin{tikzpicture}[scale=.7]
	\vertex[fill] (r3) at (1.5,.5) [label=above:\scriptsize{$v$}] {};
	\vertex[fill] (r4) at (1.5,-.5) [label=below:\scriptsize{$v+1$}] {};
    \vertex[fill] (r5) at (2.5,.5) [] {};
	\vertex[fill] (r6) at (2.5,-.5) [] {};
    \vertex[fill] (r7) at (3.5,.5) [] {};
	\vertex[fill] (r8) at (3.5,-.5) [] {};
    \vertex[fill] (r9) at (4.3,.7) [] {};
	\vertex[fill] (r10) at (4.3,-.7) [] {};
    \vertex[fill] (r11) at (5.3,.5) [] {};
	\vertex[fill] (r12) at (5.3,-.5) [] {};
\tikzstyle{vertex}=[circle, draw, inner sep=0pt, minimum size=0pt]
    \vertex (s) at (1.1,.5) [label=right:$$] {};
	\vertex (ss) at (1.1,.2) [label=right:$$] {};
    \vertex (sss) at (1.1,-.2) [label=right:$$] {};
	\vertex (ssss) at (1.1,-.5) [label=right:$$] {};
	\path
        (r3) edge (r4)
	    (r4) edge (r6)
	    (r5) edge (r3)
	    (r5) edge (r6)
	    (r5) edge (r7)
	    (r5) edge (r8)
	    (r6) edge (r7)
	    (r6) edge (r8)
	    (r12) edge (r8)
	    (r10) edge (r8)
        (r9) edge (r7)
	    (r11) edge (r7)
	    (r9) edge (r10)
	    (r9) edge (r11)
	    (r9) edge (r12)
	    (r10) edge (r11)
	    (r10) edge (r12)
	    (r11) edge (r12)
	      (r3) edge (s)
		(r3) edge (ss)
		(r4) edge (sss)
		(r4) edge (ssss);
\end{tikzpicture}\end{array}$
   \\ \hline\vspace{-.25cm}\\
 (ii)~~$\begin{array}{c}\begin{tikzpicture}[scale=.7]
	\vertex[fill] (r) at (0,0) [label=above:\scriptsize{$v$}] {};
	\vertex[fill] (r1) at (0,-1) [label=below:\scriptsize{$v+1$}] {};
\tikzstyle{vertex}=[circle, draw, inner sep=0pt, minimum size=0pt]
    \vertex (s) at (-.4,0) [] {};
	\vertex (ss) at (-.4,-.3) [] {};
    \vertex (sss) at (-.4,-.7) [] {};
	\vertex (ssss) at (-.4,-1) [] {};
	\path
		(r) edge (s)
		(r) edge (ss)
		(r1) edge (sss)
		(r1) edge (ssss);
\end{tikzpicture}\end{array}$
&  $\begin{array}{ccccc}
\begin{tikzpicture}[scale=.7]
	\vertex[fill] (r3) at (1.5,.5) [label=above:\scriptsize{$v$}] {};
	\vertex[fill] (r4) at (1.5,-.5) [label=below:\scriptsize{$v+1$}] {};
    \vertex[fill] (r5) at (2.8,.5) [label=above:\scriptsize{$v+2$}] {};
	\vertex[fill] (r6) at (2.8,-.5) [label=below:\scriptsize{$v+3$}] {};
\tikzstyle{vertex}=[circle, draw, inner sep=0pt, minimum size=0pt]
    \vertex (s) at (1.1,.5) [label=right:$$] {};
	\vertex (ss) at (1.1,.2) [label=right:$$] {};
    \vertex (sss) at (1.1,-.2) [label=right:$$] {};
	\vertex (ssss) at (1.1,-.5) [label=right:$$] {};
	\path
	    (r5) edge (r3)
	    (r4) edge (r6)
	    (r4) edge (r5)
	    (r3) edge (r6)
	     (r3) edge (s)
		(r3) edge (ss)
		(r4) edge (sss)
		(r4) edge (ssss);
\end{tikzpicture}
&&
    \begin{tikzpicture}[scale=.7]
	\vertex[fill] (r3) at (1.5,.5) [label=above:\scriptsize{$v$}] {};
	\vertex[fill] (r4) at (1.5,-.5) [label=below:\scriptsize{$v+1$}] {};
    \vertex[fill] (r5) at (2.2,0) [label=right:\scriptsize{$v+2$}] {};
    \vertex[fill] (r6) at (2.8,.5) [label=above:\scriptsize{$v+3$}] {};
	\vertex[fill] (r7) at (2.8,-.5) [label=below:\scriptsize{$v+4$}] {};
\tikzstyle{vertex}=[circle, draw, inner sep=0pt, minimum size=0pt]
    \vertex (s) at (1.1,.5) [label=right:$$] {};
	\vertex (ss) at (1.1,.2) [label=right:$$] {};
    \vertex (sss) at (1.1,-.2) [label=right:$$] {};
	\vertex (ssss) at (1.1,-.5) [label=right:$$] {};
	\path
	    (r3) edge (r5)
	    (r4) edge (r5)
	    (r5) edge (r6)
	    (r7) edge (r5)
	    (r3) edge (r6)
	    (r4) edge (r7)
	    (r3) edge (s)
		(r3) edge (ss)
		(r4) edge (sss)
		(r4) edge (ssss);
\end{tikzpicture}
&&
\begin{tikzpicture}[scale=.7]
    \vertex[fill] (r7) at (3.5,.5) [label=above:\scriptsize{$v$}] {};
	\vertex[fill] (r8) at (3.5,-.5) [label=below:\scriptsize{$v+1$}] {};
    \vertex[fill] (r9) at (4.3,.7) [] {};
	\vertex[fill] (r10) at (4.3,-.7) [] {};
    \vertex[fill] (r11) at (5.3,.5) [] {};
	\vertex[fill] (r12) at (5.3,-.5) [] {};
\tikzstyle{vertex}=[circle, draw, inner sep=0pt, minimum size=0pt]
    \vertex (s) at (3.1,.5) [label=right:$$] {};
	\vertex (ss) at (3.1,.2) [label=right:$$] {};
    \vertex (sss) at (3.1,-.2) [label=right:$$] {};
	\vertex (ssss) at (3.1,-.5) [label=right:$$] {};
	\path
	    (r12) edge (r8)
	    (r10) edge (r8)
        (r9) edge (r7)
	    (r11) edge (r7)
	    (r9) edge (r10)
	    (r9) edge (r11)
	    (r9) edge (r12)
	    (r10) edge (r11)
	    (r10) edge (r12)
	    (r11) edge (r12)
	      (r7) edge (s)
		(r7) edge (ss)
		(r8) edge (sss)
		(r8) edge (ssss);
\end{tikzpicture}\end{array}$ ~~ or returning to (i)
\\ \hline\vspace{-.25cm}\\
 (iii)~~$\begin{array}{c}\begin{tikzpicture}[scale=.7]
	\vertex[fill] (r) at (0,0) [label=above:\scriptsize{$v$}] {};
	\vertex[fill] (r1) at (0,-1) [label=below:\scriptsize{$v+1$}] {};
    \tikzstyle{vertex}=[circle, draw, inner sep=0pt, minimum size=0pt]
    \vertex (s) at (-.4,0) [] {};
	\vertex (ss) at (-.4,-1) [] {};
	\path
        (r) edge (r1)
		(r) edge (s)
		(r1) edge (ss);
\end{tikzpicture}\end{array}$
&
  $\begin{array}{c}\begin{tikzpicture}[scale=.7]
	\vertex[fill] (r) at (0,0) [label=above:\scriptsize{$v$}] {};
	\vertex[fill] (r1) at (0,-1) [label=below:\scriptsize{$v+1$}] {};
 \vertex[fill] (r2) at (1.3,0) [label=above:\scriptsize{$v+2$}] {};
	\vertex[fill] (r3) at (1.3,-1) [label=below:\scriptsize{$v+3$}] {};
     \tikzstyle{vertex}=[circle, draw, inner sep=0pt, minimum size=0pt]
    \vertex (s) at (-.4,0) [label=right:$$] {};
	\vertex (ss) at (-.4,-1) [label=right:$$] {};
	\path
        (r) edge (r1)
        (r) edge (r2)
        (r) edge (r3)
        (r1) edge (r2)
        (r1) edge (r3)
		(r) edge (s)
		(r1) edge (ss);
\end{tikzpicture}\end{array}$ ~~ or $v$ turns to a cut vertex
  \\ 		\hline
 	\end{tabular}
 	\caption{The situation of two vertical vertices and the possible structures following them}\label{tab:v,v+1}
 \end{table}

In  Case (i) it is easily seen, by switching, that $v\sim v+2$. If further $v+1 \sim v+2$, then we obtain the first possible outcome.
If $v+1\nsim v+2$, it is easily seen, by switching, that $v+1\sim v+3$.
Now, we can employ Lemma~\ref{lem:r+1--r+2} (with $H=\Gamma\setminus[v+1]$), which implies that either $v+2\sim v+3$ or either of $\tilde D'_3$ or $\tilde D'_4$ as an end block can be obtained.

 In (ii), we assume that $v\nsim v+1$, otherwise we return to Case (i).
  It is  easily seen, by switching, that $v\sim v+2$ and $v\sim v+3$. Then, $H=\Gamma\setminus[v-1]$ satisfies
  Lemma~\ref{lem:r+1,r+2VerCut} and so that either $v+1\sim v+2$, or we  obtain the third possible outcome, in which case  we either get $\tilde D'_3$  or we are left with one of the  two situations (a) or (b) of Figure~\ref{fig:thm:quarticTOmaximal}.
For (a), by  $\sw(y,v+1,x,v)$ and then $\sw(x,v+1,v,z)$, and for (b), by $\sw(x,v+1,v,z)$, we obtain $\tilde D'_4$.
  So we assume that $v+1\sim v+2$. If  we further  have $v+1\sim v+3$, we come up with the first possible outcome.  So assume that $v+1\nsim v+3$. Then it easy to find a switch that  ensures $v+1\sim v+4$. If $v+2\sim v+3$ and  $v+2\sim v+4$, then we obtain the second possible outcome. Otherwise, we have either $v+2\nsim v+3$ or $v+2\nsim v+4$, and then  $\sw(v,v+3,v+1,v+2)$ or  $\sw(v,v+2,v+1,v+4)$, respectively, ensures that $v\sim v+1$, which return us to  Case~(i).

 In (iii), it is  easily seen, by switching, that $v\sim v+2$ and $v\sim v+3$, as shown in Figure~\ref{fig:thm:quarticTOmaximalc}.
If $v+1\nsim v+2$,  then, by $\sw(x, v+1, v, v+2)$,  $v$ is turned to a  cut vertex $v$. Now, let $v+1\sim v+2$. If further $v+1\sim v+3$,  then we obtain the first outcome, otherwise by $\sw(x, v+1, v, v+3)$, $v$ is turned  into a  cut vertex $v$.
\begin{figure}[h!]
\centering
\subfloat[]
 { \begin{tikzpicture}[scale=.9]
    \vertex[fill] (r5) at (2.5,.5) [label=above:\scriptsize{$y$}] {};
	\vertex[fill] (r6) at (2.5,-.5) [label=below:\scriptsize{$x$}] {};
    \vertex[fill] (r7) at (3.5,.5) [label=above:\scriptsize{$v$}] {};
	\vertex[fill] (r8) at (3.5,-.5) [label=below:\scriptsize{$v+1$}] {};
    \vertex[fill] (r9) at (4.3,.7) [] {};
	\vertex[fill] (r10) at (4.3,-.7) [] {};
    \vertex[fill] (r11) at (5.3,.5) [label=above:\scriptsize{$z$}] {};
	\vertex[fill] (r12) at (5.3,-.5) [] {};
\tikzstyle{vertex}=[circle, draw, inner sep=0pt, minimum size=0pt]
    \vertex (s) at (2.2,.5) [label=right:$$] {};
	\vertex (ss) at (2.2,.2) [label=right:$$] {};
    \vertex (sss) at (2.2,-.2) [label=right:$$] {};
	\vertex (ssss) at (2.2,-.5) [label=right:$$] {};
	\path
	    (r5) edge (r7)
	    (r5) edge (r8)
	    (r6) edge (r7)
	    (r6) edge (r8)
	    (r12) edge (r8)
	    (r10) edge (r8)
        (r9) edge (r7)
	    (r11) edge (r7)
	    (r9) edge (r10)
	    (r9) edge (r11)
	    (r9) edge (r12)
	    (r10) edge (r11)
	    (r10) edge (r12)
	    (r11) edge (r12)
	    (r5) edge (s)
		(r5) edge (ss)
		(r6) edge (sss)
		(r6) edge (ssss);
\end{tikzpicture}\label{fig:thm:quarticTOmaximala}}
\quad  \subfloat[] {\begin{tikzpicture}[scale=.9]
	 \vertex[fill] (r5) at (2.5,.5) [] {};
	\vertex[fill] (r6) at (2.5,-.5) [label=below:\scriptsize{$x$}] {};
    \vertex[fill] (r) at (3,0) [] {};
    \vertex[fill] (r7) at (3.5,.5) [label=above:\scriptsize{$v$}] {};
	\vertex[fill] (r8) at (3.5,-.5) [label=below:\scriptsize{$v+1$}] {};
    \vertex[fill] (r9) at (4.3,.7) [] {};
	\vertex[fill] (r10) at (4.3,-.7) [] {};
    \vertex[fill] (r11) at (5.3,.5) [label=above:\scriptsize{$z$}] {};
	\vertex[fill] (r12) at (5.3,-.5) [] {};
	\path
        (r5) edge (r7)
	    (r5) edge (r)
	    (r8) edge (r)
	    (r6) edge (r)
	    (r7) edge (r)
	    (r6) edge (r8)
	    (r12) edge (r8)
	    (r10) edge (r8)
        (r9) edge (r7)
	    (r11) edge (r7)
	    (r9) edge (r10)
	    (r9) edge (r11)
	    (r9) edge (r12)
	    (r10) edge (r11)
	    (r10) edge (r12)
	    (r11) edge (r12);
\end{tikzpicture}
\label{fig:thm:quarticTOmaximalb}}
\quad \subfloat[]{\begin{tikzpicture}[scale=.9]
	\vertex[fill] (r3) at (1.5,.5) [] {};
	\vertex[fill] (r4) at (1.5,-.5) [label=below:\scriptsize{$x$}] {};
    \vertex[fill] (r5) at (2.5,.5) [label=above:\scriptsize{$v$}] {};
    \vertex[fill] (r6) at (2.5,-.5) [label=below:\scriptsize{$v+1$}] {};
	\vertex[fill] (r7) at (3.5,.5) [label=above:\scriptsize{$v+2$}] {};
    \vertex[fill] (r8) at (3.5,-.5) [label=below:\scriptsize{$v+3$}] {};
	\path
	    (r3) edge (r5)
	    (r4) edge (r6)
	    (r6) edge (r5)
	    (r5) edge (r7)
		(r5) edge (r8);
\end{tikzpicture}
\label{fig:thm:quarticTOmaximalc}}
\caption{Some of possible situations in the proof of Theorem~\ref{thm:quarticTOmaximal} }\label{fig:thm:quarticTOmaximal}
\end{figure}
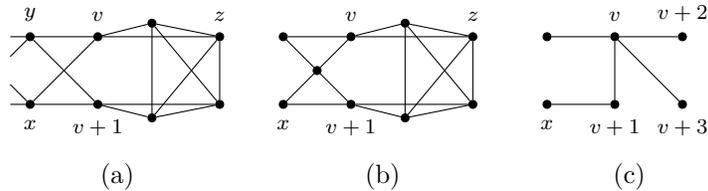

The outcome of Table~\ref{tab:v,v+1} is either an end block or, after proper reconnecting, we are again in one of the situations (i), (ii), (iii).
Therefore, we may keep repeating this until we end up with an end block.

We need further switchings to  transform the blocks into the structure asserted in Theorem~\ref{thm:quartic}.
Two types of structures may still appear in our graph: X-shape (Figure~\ref{fig:newShape}) and X$'$-shape (Figure~\ref{fig:thm:quarticTOmaximald}).
The X$'$-shape, in which $a\nsim c$ and $b\nsim d$, should be avoided. We can simply remove it by  $\sw(a,b,c,d)$, which transfers it to Figure~\ref{fig:thm:quarticTOmaximale}. For X-shapes the situation is different. They should only appear in specific places, namely in the short blocks $M,M_3,\tilde M_3$ or in an $M'_3$ or $\tilde M'_3$ as the first brick or the last brick in a long block, respectively.

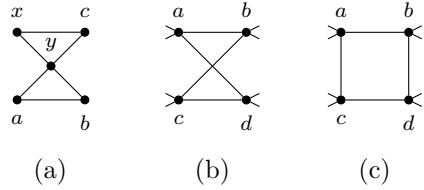
\begin{figure}[h!]
\centering
\subfloat[]{\begin{tikzpicture}[scale=.9]
	\vertex[fill] (r1) at (1.5,.5) [label=above:\scriptsize{$x$}] {};
	\vertex[fill] (r2) at (1.5,-.5) [label=below:\scriptsize{$a$}] {};
	\vertex[fill] (r3) at (2,0) [label=above:\scriptsize{$y$}] {};
	\vertex[fill] (r4) at (2.5,.5) [label=above:\scriptsize{$c$}] {};
    \vertex[fill] (r5) at (2.5,-.5) [label=below:\scriptsize{$b$}] {};
   	\path
		(r1) edge (r3)
        (r1) edge (r4)
		(r2) edge (r3)
	    (r2) edge (r5)
	    (r3) edge (r4)
	    (r3) edge (r5);
\end{tikzpicture}\label{fig:newShape}}
\qquad
 \subfloat[]{\begin{tikzpicture}[scale=.9]
	\vertex[fill] (r) at (0,0) [label=above:\scriptsize{$a$}] {};
	\vertex[fill] (r1) at (0,-1) [label=below:\scriptsize{$c$}] {};
    \vertex[fill] (r2) at (1,0) [label=above:\scriptsize{$b$}] {};
	\vertex[fill] (r3) at (1,-1) [label=below:\scriptsize{$d$}] {};
    \tikzstyle{vertex}=[circle, draw, inner sep=0pt, minimum size=0pt]
    \vertex (s) at (-.2,.1) [] {};
	\vertex (ss) at (-.2,-.1) [] {};
    \vertex (sss) at (-.2,-.9) [] {};
	\vertex (ssss) at (-.2,-1.1) [] {};
    \vertex (z) at (1.2,.1) [] {};
	\vertex (zz) at (1.2,-.1)[] {};
    \vertex (zzz) at (1.2,-.9) [] {};
	\vertex (zzzz) at (1.2,-1.1) [] {};
	\path
        (r) edge (r3)
		(r2) edge (r1)
		(r) edge (s)
		(r) edge (ss)
		(r1) edge (sss)
		(r1) edge (ssss)
	    (r2) edge (z)
		(r2) edge (zz)
		(r3) edge (zzz)
		(r3) edge (zzzz)
	    (r) edge (r2)
		(r1) edge (r3);
\end{tikzpicture}\label{fig:thm:quarticTOmaximald}}
\qquad
\subfloat[]{\begin{tikzpicture}[scale=.9]
	\vertex[fill] (r) at (0,0) [label=above:\scriptsize{$a$}] {};
	\vertex[fill] (r1) at (0,-1) [label=below:\scriptsize{$c$}] {};
    \vertex[fill] (r2) at (1,0) [label=above:\scriptsize{$b$}] {};
	\vertex[fill] (r3) at (1,-1) [label=below:\scriptsize{$d$}] {};
    \tikzstyle{vertex}=[circle, draw, inner sep=0pt, minimum size=0pt]
    \vertex (s) at (-.2,.1) [] {};
	\vertex (ss) at (-.2,-.1) [] {};
    \vertex (sss) at (-.2,-.9) [] {};
	\vertex (ssss) at (-.2,-1.1) [] {};
    \vertex (z) at (1.2,.1) [] {};
	\vertex (zz) at (1.2,-.1) [] {};
    \vertex (zzz) at (1.2,-.9) [] {};
	\vertex (zzzz) at (1.2,-1.1) [] {};
	\path
        (r) edge (r1)
		(r2) edge (r3)
		(r) edge (s)
		(r) edge (ss)
		(r1) edge (sss)
		(r1) edge (ssss)
	    (r2) edge (z)
		(r2) edge (zz)
		(r3) edge (zzz)
		(r3) edge (zzzz)
	    (r) edge (r2)
		(r1) edge (r3);
\end{tikzpicture}\label{fig:thm:quarticTOmaximale}}
\caption{X- and X$'$-shape}
\end{figure}

First note that if we have two consecutive X-shapes as in Figure~\ref{fig:Xa}, then by $\sw(a,b,c,d)$, we can  transfer it to  Figure~\ref{fig:Xd}.
If in an X-shape, the two right vertices are adjacent and it is neither in an $M$-block,  nor in an $\tilde M'_3$ (as the last brick in a long block), then it must be in the situation of Figure~\ref{fig:Xb}, which can be transferred by $\sw(a,b,c,d)$  to  Figure~\ref{fig:Xe}.
If in an X-shape, the two left vertices are adjacent and it is neither in an $M$-block, nor in an  $ M'_3$ (as the first brick in a long block), then it must be in the situation of Figure~\ref{fig:Xc}, which  can be transferred by $\sw(a,b,c,d)$ to  Figure~\ref{fig:Xf}.

\begin{figure}[h!]
\centering
\quad ~
\subfloat[]{\begin{tikzpicture}[scale=.9]
	\vertex[fill] (1) at (.5,.5) [label=above:\scriptsize{$x$}] {};
   \vertex[fill] (2) at (.5,-.5) [label=below:\scriptsize{$a$}] {};
   \vertex[fill] (3) at (1,0) [label=above:\scriptsize{$y$}] {};
	\vertex[fill] (r1) at (1.5,.5) [label=above:\scriptsize{$c$}] {};
	\vertex[fill] (r2) at (1.5,-.5) [label=below:\scriptsize{$b$}] {};
	\vertex[fill] (r3) at (2,0) [label=above:\scriptsize{$z$}] {};
	\vertex[fill] (r4) at (2.5,.5) [label=above:\scriptsize{$d$}] {};
    \vertex[fill] (r5) at (2.5,-.5) [label=below:\scriptsize{$w$}] {};
	\path
		(1) edge (r1)
		(2) edge (r2)
	    (1) edge (3)
		(2) edge (3)
	    (3) edge (r1)
		(3) edge (r2)
		(r1) edge (r4)
		(r1) edge (r3)
		(r2) edge (r3)
	    (r2) edge (r5)
	    (r3) edge (r4)
	    (r3) edge (r5);
\end{tikzpicture}\label{fig:Xa}}
\quad \quad \quad
\subfloat[]{\begin{tikzpicture}[scale=.9]
	\vertex[fill] (r1) at (1.5,.5) [label=above:\scriptsize{$x$}] {};
	\vertex[fill] (r2) at (1.5,-.5) [label=below:\scriptsize{$a$}] {};
	\vertex[fill] (r3) at (2,0) [label=above:\scriptsize{$y$}] {};
	\vertex[fill] (r4) at (2.5,.5) [label=above:\scriptsize{$c$}] {};
    \vertex[fill] (r5) at (2.5,-.5) [label=below:\scriptsize{$b$}] {};
   \vertex[fill] (r6) at (3.5,.5) [label=above:\scriptsize{$d$}] {};
   \vertex[fill] (r7) at (3.5,-.5) [label=below:\scriptsize{$z$}] {};
	\path
		(r1) edge (r3)
        (r1) edge (r4)
		(r2) edge (r3)
	    (r2) edge (r5)
	    (r3) edge (r4)
	    (r3) edge (r5)
	    (r5) edge (r4)
	     (r5) edge (r7)
	      (r4) edge (r6) ;
\end{tikzpicture}\label{fig:Xb}}
\quad \quad \quad
\subfloat[]{\begin{tikzpicture}[scale=.9]
	\vertex[fill] (1) at (.5,.5) [label=above:\scriptsize{$x$}] {};
\vertex[fill] (2) at (.5,-.5) [label=below:\scriptsize{$a$}] {};
	\vertex[fill] (r1) at (1.5,.5) [label=above:\scriptsize{$c$}] {};
	\vertex[fill] (r2) at (1.5,-.5) [label=below:\scriptsize{$b$}] {};
	\vertex[fill] (r3) at (2,0) [label=above:\scriptsize{$y$}] {};
	\vertex[fill] (r4) at (2.5,.5) [label=above:\scriptsize{$d$}] {};
    \vertex[fill] (r5) at (2.5,-.5) [label=below:\scriptsize{$z$}] {};
	\path
		(1) edge (r1)
		(2) edge (r2)
		(r1) edge (r2)
		(r1) edge (r3)
        (r1) edge (r4)
		(r2) edge (r3)
	    (r2) edge (r5)
	    (r3) edge (r4)
	    (r3) edge (r5);
\end{tikzpicture}\label{fig:Xc}}
\\
\subfloat[]{\begin{tikzpicture}[scale=.9]
	\vertex[fill] (r1) at (.5,.5) [label=above:\scriptsize{$x$}] {};
	\vertex[fill] (r2) at (.5,-.5) [label=below:\scriptsize{$a$}] {};
	\vertex[fill] (r3) at (1.5,.5) [label=above:\scriptsize{$y$}] {};
	\vertex[fill] (r4) at (1.5,-.5) [label=below:\scriptsize{$c$}] {};
    \vertex[fill] (r5) at (2.5,.5) [label=above:\scriptsize{$b$}] {};
    \vertex[fill] (r6) at (2.5,-.5) [label=below:\scriptsize{$z$}] {};
	\vertex[fill] (r7) at (3.5,.5) [label=above:\scriptsize{$d$}] {};
    \vertex[fill] (r8) at (3.5,-.5) [label=below:\scriptsize{$w$}] {};
	\path
        (r4) edge (r3)
		(r1) edge (r3)
        (r1) edge (r4)
		(r2) edge (r3)
	    (r2) edge (r4)
	    (r3) edge (r5)
	    (r4) edge (r6)
	    (r6) edge (r5)
	    (r5) edge (r7)
		(r5) edge (r8)
	    (r6) edge (r7)
	    (r6) edge (r8);
\end{tikzpicture}\label{fig:Xd}}
 \quad \quad \quad
 \subfloat[]{\begin{tikzpicture}[scale=.9]
	\vertex[fill] (r1) at (.5,.5) [label=above:\scriptsize{$x$}] {};
	\vertex[fill] (r2) at (.5,-.5) [label=below:\scriptsize{$a$}] {};
	\vertex[fill] (r3) at (1.5,.5) [label=above:\scriptsize{$y$}] {};
	\vertex[fill] (r4) at (1.5,-.5) [label=below:\scriptsize{$c$}] {};
    \vertex[fill] (r5) at (2,0) [label=above:\scriptsize{$b$}] {};
    \vertex[fill] (r6) at (2.5,.5) [label=above:\scriptsize{$d$}] {};
	\vertex[fill] (r7) at (2.5,-.5) [label=below:\scriptsize{$z$}] {};
	\path
		(r1) edge (r3)
        (r1) edge (r4)
		(r2) edge (r3)
	    (r2) edge (r4)
	    (r3) edge (r4)
	    (r5) edge (r4)
	    (r3) edge (r5)
	    (r5) edge (r6)
		(r5) edge (r7);
\end{tikzpicture}\label{fig:Xe}}
\quad \quad \quad
 \subfloat[]{\begin{tikzpicture}[scale=.9]
    \vertex[fill] (1) at (-.5,.5) [label=above:\scriptsize{$x$}] {};
	\vertex[fill] (2) at (-.5,-.5) [label=below:\scriptsize{$a$}] {};
	\vertex[fill] (r) at (0,0) [label=above:\scriptsize{$c$}] {};
	\vertex[fill] (r1) at (.5,.5) [label=above:\scriptsize{$b$}] {};
	\vertex[fill] (r2) at (.5,-.5) [label=below:\scriptsize{$y$}] {};
	\vertex[fill] (r3) at (1.5,.5) [label=above:\scriptsize{$d$}] {};
	\vertex[fill] (r4) at (1.5,-.5) [label=below:\scriptsize{$z$}] {};
	\path
        (r) edge (1)
		(r) edge (2)
		(r) edge (r1)
		(r) edge (r2)
		(r1) edge (r2)
		(r1) edge (r3)
        (r1) edge (r4)
		(r2) edge (r3)
	    (r2) edge (r4);
\end{tikzpicture}\label{fig:Xf}}
\caption{Some possible situations for  X-shapes and the results of applying $\sw(a,b,c,d)$ }
\end{figure}
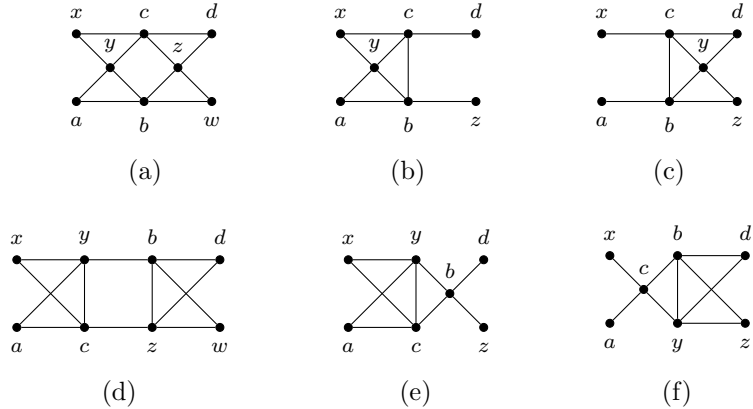

The above arguments show that $\Gamma$ can be  transformed  into one of the graphs  described in Theorem \ref{thm:quartic}.
\end{proof}

\subsubsection{Final Step}  \label{s4}

 Let $\M$ denote the family of graphs described in Theorem~$\ref{thm:quartic}$.
To complete the proof of Theorem~$\ref{thm:quartic}$, we need to show that all connected quartic graphs with minimum algebraic connectivity belong to $\M$.
In fact, it might be possible that $\Gamma$ is transformed  (by means of proper switchings) to a graph $G\in\M$, where we still have $\mu(\Gamma)=\mu(G)$.
We show that, under these circumstances, $\Gamma$  must be isomorphic to $G$.

\begin{remark}\rm Considering the structure of the graphs $G\in\M$, we regard the vertices drawn vertically above each other as a  cell.
The cells of $G$, in fact  constitute an `equitable partition' of $G$.
Each cell contains one or two vertices (except for the first cells in $D_1,D'_3,D'_4$, or some cells in the $G_i$'s (of Figure~\ref{fig:G5-G9}) that have size $4$ and $3$, respectively).
Further, we know that the weights on the vertices of $G$ given by a Fiedler vector $\rho$ of $G$ are non-increasing from left to right.
 We may assume that the vertices that are in the same cell have the same weight. Otherwise, let $\rho'$ be a vector  obtained from $\rho$ by interchanging the weights of the vertices within all cells (in fact this is carried out by the action of an automorphism of $G$, which also works for the first cells in $D_1,D'_3,D'_4$).
Then $\rho'$ and thus $\rho+\rho'$ is an eigenvector corresponding to $\mu(G)$ where $\rho+\rho'$ is  constant on  each cell.
Thus we may assume that $\rho$ is  a  non-increasing eigenvector for $\mu(G)$ and is  constant on each cell.
The above argument may not work for $G_{8'}$, but for this small graph this can be done by direct inspection.
\end{remark}

\begin{lemma}\label{lem:strictly}
Let $G\in\M$ and $\rho$  be a  non-increasing Fiedler vector of $G$  which is  constant on each cell. Then $\rho$ is indeed strictly decreasing on the cells from left to right. 
\end{lemma}
\begin{proof}
By contradiction, suppose that there are two vertices $a,b$ in two different cells with the same weight under $\rho$. We may assume that
$a\sim b$ and that at least one of $a$ or $b$ has a neighbor $c$ with  $\rho_c\ne\rho_a=\rho_b$.
 Let $\alpha$ and $\beta$ be the sum of the weights of the neighbors of $a$ and $b$, respectively.
Then, from the structure of the graphs in $\M$, it is evident that $\alpha\ge\beta$. But we have the strict inequality $\alpha>\beta$ by the existence of $c$.

We may suppose that $\|\rho\|=1$. Let $\la$ be the second largest eigenvalue of the adjacency matrix $A$ of $G$.
Then $\mu(G)=4-\la$ and $\la=\rho^\top A\rho$.
We choose a real $\epsilon$ with $0<\epsilon<(\alpha-\beta)/(1+\la)$.
Now, in the vector $\rho$ we replace the weights of $a$ and $b$ by $\rho_a+\epsilon$ and $\rho_b-\epsilon$, respectively, to obtain a new vector $\rho'$.
As $\rho\perp\bf1$, we have $\rho'\perp\bf1$.
We have
$$\la=\max_{\x\neq 0,\, \x\perp\bf1}\frac{\x ^\top A\x}{\x ^\top \x}\ge\frac{\rho'^\top A\rho'}{\rho'^\top\rho'}=\frac{\la+2\epsilon(\alpha-\beta-\epsilon)}{1+2\epsilon^2},$$
where the right hand side is larger than $\la$ by the choice of $\epsilon$, a contradiction.
\end{proof}
\begin{lemma}\label{lem:iso}
Any proper elementary move on a graph in $\M$  leaves a graph isomorphic to the original.
\end{lemma}
\begin{proof}
  For the graphs in $\M$, with a Fiedler vector which satisfies Lemma~\ref{lem:strictly}, proper switchings cannot be found  except   when $a,b$ are in the same cell, and $c,d$ are in the same cell, $a\sim c$, $b\sim d$, $a\nsim d$, and $b\nsim c$. In this case, $\sw(a, c, d, b)$ leaves a graph isomorphic to the original. Also, any proper elementary move on $G_5, G_6, G_7, G_8, G_{8^\prime}, G_9$,  and $D_1$, $D_2$, $D'_3$, and $D'_4$ gives a structure isomorphic to themselves.
 \end{proof}

 Now we can settle the `second half' of Theorem~\ref{thm:quartic}.
 The following theorem, combined with Theorem~\ref{thm:quarticTOmaximal}, completes the proof of Theorem~\ref{thm:quartic}.

\begin{theorem}\label{thm:HisoG}
 Let $\Gamma$ be a connected quartic graph such that after a sequence of proper switchings, it is turned to $G\in \M$.
   If $\mu(\Gamma)=\mu(G)$,  then $\Gamma$ is isomorphic to $G$.
\end{theorem}
\begin{proof}
Let $\sw_1,  \ldots, \sw_t$ be a sequence of proper switchings which turn $\Gamma$  into $G$.
Consider the graphs $\Gamma=G_0, G_1, \ldots, G_t=G$ in which $G_i$ is obtained from $G_{i-1}$ by applying $\sw_i$.
Since $\mu(\Gamma)=\mu(G)$, we have $\mu(G_i)=\mu(G)$,  for $i=1,\ldots, t$.
Let $\sw_t=\sw(a, b, c, d)$. Then
$$0=\mu(G_{t-1})-\mu(G)\geq \rho^\top L(G_{t-1})\rho-\rho^\top L(G)\rho= 2(\rho_a-\rho_d)(\rho_c-\rho_b)\geq 0.$$
It follows that  $\rho_a=\rho_d$ or $\rho_c=\rho_b$.
Without loss of generality, suppose that $\rho_a=\rho_d$. From Lemma~\ref{lem:strictly}  it then  follows that
$a,d$ are in the same cell  of $G$. Note that $\sw(d, b, c, a)$ is the reverse of $\sw(a, b, c, d)$, and so,  when  applied on $G$,  yields  $G_{t-1}$.
However, $\sw(d, b, c, a)$ is indeed a proper switching, and so by Lemma~\ref{lem:iso}, $G_{t-1}$ must be isomorphic to $G$.
 Similarly, it follows that all $G_i$, for $i=0,\ldots,t-2$, are isomorphic to $G$.
\end{proof}

\subsection{Concluding Remarks}
By Theorem~\ref{thm:quartic} it can be seen that the connected quartic graphs on $n \leq 10$ vertices  with minimum spectral gap are $G_5, G_6, G_7, G_8, G_9$, and the graph of Figure~\ref{fig:lemma3.9.8c}, respectively.
For $n\ge11$, we pose the following conjecture on the puniness and the precise structure of
 the  connected quartic graphs   with minimum spectral gap. The conjecture suggests that in such graphs
 all middle blocks are $M_1$ and end blocks are one of the short blocks $D_1,D_2,D_4$ or the block $D_5$ given in Figure~\ref{fig:BlockMin}.
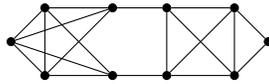
\begin{figure}[h!]
\centering
\begin{tikzpicture}[scale=0.9]
	\vertex[fill] (r1) at (0,0) [] {};
	\vertex[fill] (r3) at (.5,.5) [] {};
	\vertex[fill] (r2) at (.5,-.5) [] {};
	\vertex[fill] (r5) at (1.5,.5) [] {};
	\vertex[fill] (r4) at (1.5,-.5) [] {};
\vertex[fill] (r6) at (2.3,-.5) [] {};
\vertex[fill] (r7) at (2.3,.5) [] {};
\vertex[fill] (r8) at (3.3,-.5) []{};
\vertex[fill] (r9) at (3.3,.5) [] {};
\vertex[fill] (r10) at (3.8,0) [] {};
	\path
		(r1) edge (r2)
		(r1) edge (r3)
	   (r1) edge (r4)
	   (r1) edge (r5)
		(r2) edge (r3)
		(r2) edge (r4)
        (r2) edge (r5)
		(r3) edge (r4)
	    (r3) edge (r5)
	    (r6) edge (r4)
		(r5) edge (r7)
	    (r6) edge (r7)
		(r6) edge (r8)
	    (r6) edge (r9)
		(r8) edge (r7)
	    (r7) edge (r9)
		(r8) edge (r9)
    	(r8) edge (r10)
		(r9) edge (r10);
\end{tikzpicture}
\caption{The block $D_5$}\label{fig:BlockMin}
\end{figure}

 \begin{conjecture}\label{conj:MinQuartic}\rm  The connected quartic graph on $n\ge11$ vertices with minimum spectral gap is the unique graph $G$ described below.
Let $q$ and $r<5$ be non-negative integers such that $n-11=5q+r$.
Then $G$ consists of $q$ middle blocks $M_1$ and each end block is one of $D_1,D_2,D_4$, or $D_5$. If $r=0$, then both end blocks are $D_4$. If $r=1$, then  the end blocks are $D_4$ and $D_1$. If $r=2$, then both end blocks are $D_1$. If $r=3$, then  the end blocks are $D_1$ and $D_2$.  Finally, if  $r=4$, then  the end blocks are $D_4$ and $D_5$. 
\end{conjecture}


\section*{Acknowledgment}
The research of the  second author was in part supported by a grant from IPM (No. 98050211).
The authors would like to thank anonymous referees for constructive comments which led to improvement of the presentation of the paper.


\begin{thebibliography}{9}
\bibitem{actt} S.G. Aksoy, F.R. Chung, M. Tait, and J. Tobin, The maximum relaxation time of a random walk, {\em Adv. in Appl. Math.} {\bf101} (2018), 1--14.
\bibitem{aldous2002reversible} D. Aldous and J. Fill,  {\em Reversible Markov Chains and Random Walks on Graphs}, University of California, Berkeley,  2002, available at \url{http://www.stat.berkeley.edu/~aldous/RWG/book.html}
\bibitem{Imrich} C. Brand, B. Guiduli, and W. Imrich, The characterization of cubic graphs with minimal eigenvalue
gap, {\em Croatica Chemica Acta} {\bf80} (2007), 193--201.
\bibitem{Bussemaker} F.C. Bussemaker, S. \v Cobelji\'c, D.M. Cvetkovi\' c, and J.J. Seidel, Computer investigation of cubic graph, Technical Report No. 76-WSK-01,
Technological University Eindhoven, (1976).
\bibitem{Bussemaker2} F.C. Bussemaker, S. \v Cobelji\'c, D.M. Cvetkovi\' c, and J.J. Seidel, Cubic graphs on $\leq 14$  vertices,
{\em J. Combin. Theory Ser. B} {\bf23} (1977), 234--235.
\bibitem{chung} F.R. Chung, {\em Spectral Graph Theory}, vol. 92, American Mathematical Society, 1997.
\bibitem{fiedler1973algebraic}
M. Fiedler, Algebraic connectivity of graphs, {\em Czechoslovak Math. J.} {\bf23} (1973), 298--305.
\bibitem{GuiduliThesis} B. Guiduli, {\em Spectral Extrema for Graphs},  Ph.D. Thesis,  University of Chicago, 1996.
\bibitem{Guiduli} B. Guiduli,  The structure of trivalent graphs with minimal eigenvalue gap, {\em J. Algebraic Combin.} {\bf6} (1997), 321--329.
\end{thebibliography}
\end{document}